\documentclass[12pt,oneside,reqno]{amsart}

\usepackage[a4paper,left=28mm,right=28mm,top=27mm,bottom=30mm]{geometry}
\usepackage{amsmath,amssymb,amsthm,mathtools,mathrsfs}
\usepackage{enumitem}
\usepackage{xcolor}
\usepackage[colorlinks=true,linkcolor=blue!55!black,citecolor=blue!55!black,
            urlcolor=blue!55!black]{hyperref}
\usepackage[nameinlink,noabbrev]{cleveref}
\hypersetup{
  pdftitle={Grothendieck Topologies Are Extensional Presentations of the Form of Sieves},
  pdfauthor={Roy Ferguson and Zurab Janelidze}
}

\allowdisplaybreaks
\setlist{leftmargin=2.2em}

\newtheorem{theorem}{Theorem}[section]
\newtheorem{lemma}[theorem]{Lemma}

\theoremstyle{definition}
\newtheorem{definition}[theorem]{Definition}
\newtheorem{example}[theorem]{Example}
\theoremstyle{remark}

\crefname{theorem}{Theorem}{Theorems}
\crefname{lemma}{Lemma}{Lemmas}
\crefname{corollary}{Corollary}{Corollaries}
\crefname{construction}{Construction}{Constructions}
\crefname{definition}{Definition}{Definitions}
\crefname{example}{Example}{Examples}
\crefname{remark}{Remark}{Remarks}

\newcommand{\C}{\mathcal C}
\newcommand{\MFrm}{\mathbf{MFrm}}
\newcommand{\DFrm}{\mathbf{DFrm}}
\newcommand{\SDFrm}{\mathbf{SDFrm}}
\newcommand{\Null}{\mathcal N}
\newcommand{\Meet}{\mathbf{Meet}_{\top}}
\newcommand{\Sieve}{\mathscr S}
\newcommand{\Var}{\mathscr V}

\newcommand{\supp}[2][]{\lvert #2\rvert_{#1}}
\newcommand{\prin}[1]{\langle #1\rangle}
\newcommand{\Sat}{\mathscr L}
\newcommand{\Sub}{\mathrm{Sub}}
\newcommand{\Idl}{\mathrm{Idl}}
\newcommand{\CRing}{\mathbf{CRing}}

\title[Grothendieck Topologies Are Extensional Presentations]
{Grothendieck Topologies Are Extensional Presentations of the Form of Sieves}

\author{Roy Ferguson}
\address{Artificial Intelligence Research Unit, Department of Computer
Science, University of Cape Town, South Africa}
\email{roy.ferguson@uct.ac.za}

\author{Zurab Janelidze}\thanks{\textbf{AI Assistance Disclosure}. The second author used Chat GPT-6 Astra to assist with the literature review, research, and writing and editing of this paper. The mathematical work is human-driven and AI-assisted: the authors identified the central ideas, formulated the research direction, and decided on the course of research throughout engaging with the AI, which was used solely as a time-saving aid to support research progress and improve the efficiency of writing. All AI-assisted material was critically reviewed and verified by the authors, who take full responsibility for the paper’s content, accuracy, and integrity.}
\address{Mathematics Division, Stellenbosch University, South Africa;
National Institute for Theoretical and Computational Sciences (NITheCS)}
\email{zurab@sun.ac.za}

\date{First draft, September 2026}


\begin{document}

\begin{abstract}
A Grothendieck topology on a category determines both a subform of
covering sieves and a quotient form obtained by identifying locally
equivalent sieves.  We place these two constructions in a single short
exact sequence.  For this purpose we introduce distributivity forms:
indexed meet-semilattices equipped with distinguished indexed joins
over which finite indexed meets distribute, and a strong version in
which these joins also satisfy Beck--Chevalley.  Over a fixed base, the
resulting categories have all small limits and have kernels and
cokernels relative to the closed ideal of fibrewise constant-top
morphisms.  Their monomorphisms are the fibrewise injective morphisms,
and their relative cokernels, together with the top-reflecting
morphisms, form an orthogonal factorization system.  Cokernels compose
but need not be stable under pullback, whereas kernels need not compose.
We call a short exact sequence with prescribed middle term an
extensional presentation, and prove that Grothendieck topologies on a
category are precisely the extensional presentations of its maximally
distributive form of sieves.  Further applications recover universal
productive closure operators and Lawvere--Tierney topologies,
functorial non-Archimedean group topologies, and functorial linear
topologies on commutative rings.
\end{abstract}

\maketitle

\section*{Introduction}

A \emph{form} is a category equipped with an abstract data of structures at an object. These structures can be subobjects, quotients, a combination of the two (e.g., subquotients), sieves, and so no. A form abstracts the situations where such structures form posets at each object. A form is defined as a faithful amnestic functor to the base category, whose fibres are such posets. We refer to the elements in the fibres as \emph{clusters} of the form. We can also view forms as indexed posets, via a variant of the Grothendieck construction \cite[Section~B1.3]{JohnstoneElephant}.

Forms have been used to characterize cover relations,
relative subobject classes, factorization systems, and exactness
structures \cite{JanelidzeWeighillI,JanelidzeWeighillII,
JanelidzeWeighillIII}. Doctrines studied in categorical logic are also forms
\cite{LawvereAdjointness,Jacobs,JohnstoneElephant}.
Forms also provide a natural setting for a
unified treatment of categorical closure operators
\cite{DuckertsAntoineGranJanelidze}. Noetherian forms introduced in \cite{GoswamiJanelidze} 
(see also \cite{JanelidzeSubgroups,VanNiekerkBiproducts,
VanNiekerkConstruction,JanelidzeVanNiekerk}) offer a  
particularly instructive precedent. Normal and conormal clusters there
satisfy self-dual axioms, and the usual homomorphism and Noether
isomorphism theorems can consequently be obtained without building into
the language the asymmetry between subobjects and quotients. The fact that a topos has
a noetherian form (see \cite{JanelidzeVanNiekerk}) shows that this use of forms is not confined
to algebraic categories.

This paper grew out from one of the chapters of the PhD thesis of the first author
\cite{FergusonThesis}, where he characterizes Grothendieck topologies as
particular types of forms. There are two approaches here. On the one
hand, the covering sieves of a Grothendieck topology constitute a distinguished
subform of the form of all sieves. On the other hand, the covering
relation arising from a Grothendieck topology \cite[Chapter~III, Section~2]{MacLaneMoerdijk} can be used to order the
corresponding equivalence classes of sieves, which results in a form whose clusters are generalized subobjects determined by the topology. In both
cases, the characterization makes strong use of the form of sieves. In
fact, the form of sieves appears to be an extension of the two forms
arising in these characterizations. To make this observation precise,
we set out to uncover the category in which such extensions could be represented as short exact sequences. 

It turned out that the category we need is that of what we call in this paper
\emph{distributivity forms}. These are forms where all finite indexed meets exist and which come with a specified class of indexed joins such that the meet
distributes over them.  The same data can equivalently be
described in terms of initial and final lifts of a form seen as a  functor.
This connects distributivity forms with the theory of topological functors developed by Br\"ummer and others
\cite{Brummer,HerrlichTopological,BrummerHoffmann}. Both
descriptions are useful: lifts explain the categorical origin of the operations,
while indexed meets and joins provide a convenient calculus for them.
We call a distributivity \emph{strong} when it
also satisfy a suitable Beck--Chevalley condition. The distributivity law is a
Frobenius condition, controlling stability along vertical arrows in the
total category, while Beck--Chevalley controls stability along cartesian
arrows.  When the base category has pullbacks, the axioms of strong
distributivity make the families of total arrows occurring in the
selected indexed joins into a Grothendieck pretopology, and hence
generate a Grothendieck topology on the domain of the form viewed as a functor.

We show that over a fixed base category, the categories of distributivity forms and
of strong distributivity forms have all small limits.  Their
monomorphisms are precisely the fibrewise injective morphisms and are
therefore fibrewise order embeddings.  The morphisms which are
fibrewise constant at the top cluster form a closed ideal, although
there need not be such a morphism between every pair of forms.  Relative
to this ideal, kernels and cokernels exist. These categories are therefore semiexact in the sense of Grandis
\cite{GrandisRelative,GrandisBook}. They need not, however, be
homological. In both cases, cokernels are closed under composition but need not be
stable under pullback. Kernels,
in turn, need not be closed under composition. Also, while cokernels are part of a factorization system, this factorisation system is not proper. Overall, this means that the exactness structure of these categories is, unfortunately, not particularly well behaved. Yet, they are useful for solving the problem that we set out to solve.

By an \emph{extensional
presentation} of a distributivity form $\mathcal{F}$ we mean a short exact sequence in the category of distributivity forms, where $\mathcal{F}$ is the middle object. We prove that Grothendieck topologies on a category \(\C\) are in bijection with
extensional presentations of its maximally distributive form of sieves. The latter form arises as the extension of the two ways of viewing a Grothendieck topology as a form, discussed above. In particular, this means that Grothendieck topologies can be characterized as normal subobjects (and equivalently, normal quotients) of the form of sieves.

Next, we set out to explore if there are interesting extensional presentations of other forms, different from the form of sieves. We show that when \(\C\) has pullbacks, extensional presentations of the unary
distributivity on the form of variations classify singleton-generated
Grothendieck topologies.  For a locally small, well-powered geometric
category, extensional presentations of the geometric subobject form are
universal productive closure operators; in the presence of a subobject
classifier, these are the familiar Lawvere--Tierney topologies. Over
the terminal category, a distributivity form is a meet-semilattice
equipped with selected exact joins, and the theory connects with
semilattice sites, partial frames, filters, and fitted nuclei.

Further to the examples above, we show that extensional presentations
of the minimally distributive subgroup form are precisely functorial
non-Archimedean group topologies, including the profinite and pro-\(p\)
constructions.  Extensional presentations of the minimally distributive
ideal form are precisely functorial linear topologies on commutative
rings. These include profinite and adic examples, as well as principal
topologies attached to functorial ideal assignments such as the
nilradical.

The results of the present paper leaves one wondering whether the notion of an extensional presentation of a distributivity form is a useful generalization of the notion of a Grothendieck topology, which replaces the form of sieves with a general distributivity form. A good test of this usefulness would be development of sheaf theory relative to distributivity forms. This, however, falls beyond the scope of the present paper: we only make some preliminary steps in this direction at the very end of the paper.

The paper is organized as follows.  Section~1 introduces distributivity
forms and
studies strong and generated distributivities, explains how strong
distributivities produce topologies on their total categories, and gives
intrinsic characterizations of various relevant forms.  Section~2 introduces and studies the categories of distributivity forms over a base category. It also develops briefly a general theory of extensional
presentations. Section~3
proves the classification theorems and exhibits and studies various extensional presentations of familiar forms.
Section~4 makes some further remarks that could be of interest to the reader, but do not fall strictly within the main programme of the paper.

Throughout, the base category \(\C\) is assumed to be small.  As usual,
this assumption may instead be interpreted relative to a fixed choice
of universes.

\begingroup
\footnotesize
\tableofcontents
\endgroup

\section{Distributivity forms}
\label{sec:distributivity-forms}

\subsection{Indexed meets, indexed joins, and lifts}

We use the language of forms developed in
\cite{JanelidzeWeighillI,JanelidzeWeighillII}; see also
\cite{DuckertsAntoineGranJanelidze,JanelidzeVanNiekerk}.  A
\emph{form} over \(\C\) is a faithful amnestic functor
\[
                         p:\mathcal E\longrightarrow\C.
\]
Its objects are called \emph{clusters}.  The fibre \(F_X\) over an
object \(X\) is a poset: for \(A,B\in F_X\), we write \(A\leqslant B\)
when there is an arrow \(A\to B\) above \(1_X\).  More generally,
\(A\leqslant_fB\) means that the unique possible arrow \(A\to B\)
above \(f:pA\to pB\) exists.

We first recall the two lifting notions which organize the definition.  A
\emph{structured source} in \(p\), with base domain \(X\), consists of
clusters \(A_i\) and arrows
\[
                         f_i:X\longrightarrow pA_i.
\]
A lift of the source is a cluster \(A\in F_X\) with
\(A\leqslant_{f_i}A_i\) for every \(i\).  It is \emph{initial} when,
for every cluster \(B\) and every \(g:pB\to X\),
\[
 B\leqslant_g A
 \quad\Longleftrightarrow\quad
 B\leqslant_{f_i g}A_i\quad\text{for every }i.
\]
A \emph{structured sink} with base codomain \(X\) consists of clusters
\(A_i\) and arrows
\[
                         f_i:pA_i\longrightarrow X.
\]
A lift is a cluster \(U\in F_X\) with
\(A_i\leqslant_{f_i}U\) for every \(i\), and it is \emph{final} when,
for every cluster \(B\) and every \(g:X\to pB\),
\[
 U\leqslant_gB
 \quad\Longleftrightarrow\quad
 A_i\leqslant_{g f_i}B\quad\text{for every }i.
\]
These are the standard initial and final lifting conditions for a
faithful functor.  They originate in categorical topology; our use of
them follows in particular Br\"ummer's treatment of topological functors \cite{Brummer} (see also
\cite{HerrlichTopological,BrummerHoffmann,GarnerTopological}).  Cartesian
and cocartesian lifts are the singleton cases of initial and final lifts,
respectively.

Suppose that all finite structured sources have initial lifts.  The
empty source has an initial lift \(\top_X\), a singleton source
\(f:X\to Y\), \(B\in F_Y\), has an initial lift denoted \(f^*B\), and
the two-object source consisting of \(A,B\in F_X\), with both indexing
arrows equal to \(1_X\), has an initial lift \(A\wedge B\).  Thus the
fibres have top and binary meets, every
inverse image exists, and
\[
 f^*\top_Y=\top_X,
 \qquad
 f^*(A\wedge B)=f^*A\wedge f^*B.
\]
Conversely, these operations supply every finite initial lift, by
\[
              \bigwedge_i f_i^*A_i.
\]
Thus a finite initial lift is a \emph{finite indexed meet}.  Equivalently,
after choosing inverse images, the form is an indexed meet-semilattice
\begin{equation}\label{eq:indexed-meet-form}
                  F:\C^{\mathrm{op}}\longrightarrow\Meet,
\end{equation}
where \(\Meet\) denotes meet-semilattices with top and
finite-meet-preserving maps.  The ordered Grothendieck construction
recovers the total category: an arrow \((X,A)\to(Y,B)\) over
\(f:X\to Y\) exists precisely when \(A\leqslant f^*B\).  This is the
indexed-poset description of the same form, not additional structure. In general, even when inverse images do not exist, we can write the relation \(A\leqslant f^*B\) as \(A\leqslant_f B \) or equivalently, \(B\geqslant_f A \) and still think of a form as an indexed poset in a less convenional sense (see \cite{DuckertsAntoineGranJanelidze,JanelidzeVanNiekerk}).

For a structured sink, a final lift is equivalently the
least \(U\in F_X\) satisfying \(A_i\leqslant f_i^*U\) for every
\(i\).  We write it as
\[
             U=\nabla_i(A_i,f_i)
               =\bigvee_i{}^{f_i} A_i
\]
and call it an \emph{indexed join}.  When the singleton indexed join
along \(f\) exists it is the cocartesian direct image \(f_!A\), left
adjoint to \(f^*\); if all the relevant direct images exist, the formula
becomes \(\bigvee_i f_{i!}A_i\).  When all \(f_i\) are identities,
indexed meets and joins are ordinary meets and joins in one fibre.  The
whole displayed datum, not merely its apex \(U\), will be called a
\emph{final presentation}.  We use ``indexed join'' when emphasizing
the operation and ``final presentation'' when emphasizing its specified
source and universal property.

\begin{definition}\label{def:distributivity-form}
A \emph{distributivity form} over \(\C\) is a form in which every
finite structured source has an initial lift, together with a selected
class \(\mathcal P_F\) of small final presentations satisfying the
following conditions.
\begin{enumerate}[label=\textup{(D\arabic*)}]
\item Every unit presentation
\(A=\nabla(A,1_X)\) is selected, and selection is invariant under
bijective reindexing.
\item Selected final lifts compose: if
\[
 U=\nabla_i(A_i,f_i),
 \qquad
 A_i=\nabla_j(B_{ij},g_{ij})
\]
are selected, then so is
\[
                 U=\nabla_{i,j}(B_{ij},f_i g_{ij}).
\]
\item Finite indexed meets distribute over selected indexed joins.  In
the elementary notation above, if
\(U=\nabla_i(A_i,f_i)\) is selected and \(C\in F_X\), then
\[
 U\wedge C
   =\nabla_i(A_i\wedge f_i^*C,f_i)
\]
is again a selected final presentation.
\end{enumerate}
The selected class \(\mathcal P_F\) is called the
\emph{distributivity} of \(F\), and its members are called
\emph{admissible presentations}.
\end{definition}

Condition (D2) is substitution of indexed joins.  Its displayed equality
is forced by the two universal properties; the content is closure of
the selected class.  Condition (D3) is a Frobenius law.  Since every
finite initial lift is a finite meet of inverse images, it says exactly
that the chosen final operations distribute over the finite initial
operations at their common apex.

There is a second compatibility, independent of Frobenius.  Assume for
the moment that \(\C\) has pullbacks.  Given an admissible presentation
\(U=\nabla_i(A_i,f_i)\) and an arrow \(h:Y\to X\), form pullback squares
\[
\begin{array}{ccc}
Y_i&\xrightarrow{\ h_i\ }&X_i\\[-2pt]
{\scriptstyle f_i'}\!\downarrow&&\downarrow\!{\scriptstyle f_i}\\[-2pt]
Y&\xrightarrow{\ h\ }&X .
\end{array}
\]

\begin{definition}\label{def:strong-distributivity}
A distributivity is \emph{strong} when it also satisfies:
\begin{enumerate}[label=\textup{(D4)},leftmargin=4.2em]
\item the Beck--Chevalley transform of every admissible presentation is
again admissible, and
\begin{equation}\label{eq:beck-chevalley}
 h^*U=\nabla_i(h_i^*A_i,f_i')
       =\bigvee_i{}^{f_i'}h_i^*A_i.
\end{equation}
\end{enumerate}
A distributivity form equipped with a strong distributivity is a
\emph{strong distributivity form}.  If the base does not have all
pullbacks, (D4) is read for the displayed pullbacks which exist; a
pullback-refinement version may instead be used when a coverage, rather
than a pretopology, is desired.
\end{definition}

Frobenius and Beck--Chevalley have different roles.  Condition (D3)
pulls an admissible family back along a vertical arrow in the total
category, while (D4) pulls it back along a cartesian arrow.  Every arrow
in the total category factors as a vertical arrow followed by a
cartesian one, so together they express stability along arbitrary total
arrows.

\begin{definition}\label{def:exact-presentation}
A final presentation is \emph{exact} when all the presentations required
from it by (D3) exist and are final.  It is \emph{universally exact} when
it is exact and all its Beck--Chevalley transforms
\eqref{eq:beck-chevalley} exist and are exact.
\end{definition}

\begin{lemma}\label{lem:generated-distributivity}
Every class \(\mathcal A\) of exact final presentations is contained in
a least distributivity, obtained by closing \(\mathcal A\) under units,
bijective reindexing, substitution, and (D3).  Every class of universally
exact presentations is similarly contained in a least strong
distributivity, obtained by also closing under (D4).
\end{lemma}

\begin{proof}
Intersections of classes satisfying (D1)--(D3), or (D1)--(D4), satisfy
the same conditions.  The class of all exact presentations is a
distributivity, and the class of all universally exact presentations is
a strong distributivity, so the relevant intersections exist.  This is
equivalently the closure obtained by iterating the displayed rules.
There is generally no fixed finite number of iterations: substitution
produces presentations of arbitrarily large finite depth.  For finitary
rules the union of the finite stages gives the closure; allowing
arbitrary small substitutions may require the usual transfinite closure.
\end{proof}

The selected presentations also define covering families.  Recall that
a Grothendieck topology assigns covering sieves, contains each maximal
sieve, and is stable under pullback and local refinement.  Equivalently,
a system of identity and composite covering families which is stable
under pullback is a Grothendieck pretopology; it generates a
Grothendieck topology.

\begin{theorem}\label{thm:strong-total-site}
Let \(\C\) have pullbacks and let \(p:\mathcal E\to\C\) be a strong
distributivity form.  The total arrows
\[
                  A_i\longrightarrow U
\]
belonging to its admissible final presentations form a Grothendieck
pretopology on \(\mathcal E\), and hence generate a Grothendieck topology
there.  This pretopology is subcanonical precisely when the base family
\((f_i:pA_i\to pU)_i\) of every admissible presentation covers for the
canonical topology of \(\C\).
\end{theorem}

\begin{proof}
Condition (D1) gives identity covers and (D2) gives their substitution.
By the factorization of total arrows just described, (D3) and (D4) give
pullback stability.  Thus the families form a pretopology.  Finality
says that, once a common base map out of \(pU\) is given, it lifts out of
\(U\) exactly when its composites lift out of every \(A_i\).  Descent
for representables therefore reduces to unique gluing of those base
maps, which is exactly the canonical-cover condition on \((f_i)_i\).
Necessity may already be tested against the top clusters, since
\[
 \operatorname{Hom}_{\mathcal E}(A,\top_Y)
     \cong\operatorname{Hom}_{\C}(pA,Y).
\]
\end{proof}

Thus strong distributivity supplies a Grothendieck topology on the
domain of the form, but not automatically a subcanonical one.  For
example, for the identity form on a category with pullbacks and the
largest possible class of presentations, (D1)--(D4) are automatic,
whereas declaring every family covering need not be subcanonical.

The distinction between a join and its presentation is important.  A
cluster may have several presenting sinks, and morphisms introduced in
Section~\ref{sec:extensional-presentations} will be required to preserve each selected
sink, not only the equality of its apex with some join.

\subsection{The form of sieves}

Let \(\Sieve_X\) be the poset of sieves on \(X\), ordered by inclusion.
For \(f:X\to Y\), set
\[
 f^*S=\{g:Z\to X\mid fg\in S\}.
\]
The maximal sieve is the top cluster, intersections are meets, and the
displayed operation is inverse image.  These data form the
\emph{form of sieves} \(\Sieve\) over \(\C\).

For a sieve \(R\) on \(Y\) and an arrow \(f:Y\to X\), write
\[
                         fR=\{fg\mid g\in R\}.
\]
This is a sieve on \(X\).

\begin{lemma}\label{lem:sieve-distributivity}
Every structured sink in \(\Sieve\) has a final lift, given by
\begin{equation}\label{eq:sieve-final-lift}
                   \nabla_i(S_i,f_i)=\bigcup_i f_iS_i.
\end{equation}
Taking all these final presentations makes \(\Sieve\) a distributivity
form.  If \(\C\) has pullbacks, this maximal distributivity is strong.
\end{lemma}

\begin{proof}
A sieve \(T\) is an upper bound of the structured sink precisely when
\(f_iS_i\subseteq T\) for every \(i\), so
\eqref{eq:sieve-final-lift} is its final lift.  Units and substitution
follow from identities and associativity of composition.  The
distributivity law is the set-theoretic identity
\[
 C\cap\bigcup_i f_iS_i
   =\bigcup_i f_i(S_i\cap f_i^*C).
\]
For a pullback square as in Definition
\ref{def:strong-distributivity}, inverse image of a generated sieve gives
\[
 h^*\!\left(\bigcup_i f_iS_i\right)
   =\bigcup_i f_i'h_i^*S_i,
\]
which proves (D4).
\end{proof}

We call this choice the \emph{maximal distributivity} of \(\Sieve\).
Its final lifts are the basic mechanism by which the local character of
a Grothendieck topology will become an exactness statement.

For a cluster \(A\in F_X\) of any form with finite indexed meets, put
\begin{equation}\label{eq:support-sieve-early}
 \supp[F]{A}=\{f:Y\to X\mid f^*A=\top_Y\}.
\end{equation}
This is a sieve, called the \emph{support} of \(A\).  The assignments
\(A\mapsto\supp[F]{A}\) preserve top, finite meets, and inverse images,
and therefore define a morphism of the underlying indexed meet-posets
\[
                     \sigma_F:F\longrightarrow\Sieve.
\]

\begin{theorem}\label{thm:characterize-sieve-form}
Let \(F\) be a form over \(\C\) with finite indexed meets.
\begin{enumerate}[label=\textup{(\roman*)}]
\item The underlying form \(F\) is isomorphic to the form of sieves if
and only if clusters are determined by their supports and every sieve
is the support of a cluster; equivalently, \(\sigma_F\) is fibrewise
bijective.
\item A distributivity form \(F\) is isomorphic to the maximally
distributive form of sieves if and only if the conditions in \textup{(i)}
hold and every structured sink has an admissible final lift.
\end{enumerate}
When \(\C\) has pullbacks, the object in \textup{(ii)} is a strong
distributivity form.
\end{theorem}

\begin{proof}
For sieves, \(f^*S=\top\) holds exactly when \(f\in S\); hence the
support map of \(\Sieve\) is the identity.  In general,
\[
 \supp[F]{A\wedge B}=\supp[F]{A}\cap\supp[F]{B},\qquad
 \supp[F]{g^*A}=g^*\supp[F]{A},
\]
and top has maximal support.  Thus \(\sigma_F\) is a morphism of the
underlying forms.  It is an isomorphism exactly under the two conditions
in \textup{(i)}.  An isomorphism of forms preserves and reflects final
lifts.  Consequently its distributivity is transported to the maximal
one exactly when every structured sink is admissible.  The last
assertion is Lemma~\ref{lem:sieve-distributivity}.
\end{proof}

\subsection{The two forms determined by a Grothendieck topology}

Recall that a Grothendieck topology \(J\) assigns to each \(X\) a
collection \(J(X)\) of covering sieves such that the maximal sieve
covers, inverse images of covering sieves cover, and the following
locality condition holds: if \(B\in J(X)\) and \(f^*R\) covers for
every \(f\in B\), then \(R\in J(X)\).  Covering sieves are upward
closed and closed under finite intersections.

The topology first determines the \emph{covering form} whose fibre over
\(X\) is \(J(X)\).  Its distributivity is induced from \(\Sieve\): a
presentation
\[
                         U=\bigcup_i f_iS_i
\]
is admissible when all the source sieves \(S_i\) and the apex \(U\)
are covering.

There is also a quotient-shaped form.  Following the covering-relation
formulation of Mac~Lane and Moerdijk
\cite[Chapter~III, Section~2]{MacLaneMoerdijk}, say that a sieve \(S\)
on \(X\) \emph{covers} an arrow \(f:Y\to X\) when \(f^*S\in J(Y)\).
The sieve of arrows covered by \(S\) is
\[
 j_{J,X}(S)=\{f:Y\to X\mid f^*S\in J(Y)\}.
\]
Two sieves are \emph{locally equivalent} when they cover the same
arrows, or equivalently when their images under \(j_J\) agree.  A
sieve fixed by \(j_J\) will be called \(J\)-\emph{saturated}.  Write
\(\Sat_J(X)\) for the poset of saturated sieves on \(X\).

\begin{theorem}\label{thm:topology-gives-two-forms}
Every Grothendieck topology \(J\) on \(\C\) determines two
distributivity forms over \(\C\):
\begin{enumerate}[label=\textup{(\roman*)}]
\item the covering form \(J\), with the distributivity induced from
the maximal distributivity of \(\Sieve\);
\item the saturation form \(\Sat_J\), whose admissible final lifts are
\begin{equation}\label{eq:saturation-final-lift}
 \nabla_i(T_i,f_i)
   =j_J\!\left(\bigcup_i f_iT_i\right),
 \qquad T_i\in\Sat_J(X_i).
\end{equation}
\end{enumerate}
Each local-equivalence class of sieves has the largest representative
\(j_J(S)\), and these representatives are precisely the clusters of
\(\Sat_J\).  If \(\C\) has pullbacks, both distributivity forms are
strong.
\end{theorem}

\begin{proof}
The usual consequences of the Grothendieck topology axioms show that
the covering sieves contain top and are closed under inverse image and
finite intersection.  They therefore admit all finite initial lifts.
An admissible presentation induced from \(\Sieve\) remains a final lift
inside \(J\), since its apex is already the least upper bound among all
sieves.  Conditions (D1) and (D2) are inherited immediately.  For
(D3), if \(C\) and all \(S_i\) cover, then
\(S_i\cap f_i^*C\) covers; if the original apex \(U\) covers, then so
does \(U\cap C\).  Hence the distributive presentation from
Lemma~\ref{lem:sieve-distributivity} remains internal to \(J\).

For the second construction, the operators \(j_{J,X}\) are extensive
and monotone.  They preserve binary intersections because a finite
intersection of sieves covers exactly when both factors cover.  Directly
from the definition,
\begin{equation}\label{eq:saturation-base-change}
              f^*j_{J,Y}(S)=j_{J,X}(f^*S).
\end{equation}
Moreover, \(j_J(S)\) covers if and only if \(S\) covers.  One direction
uses extensivity and upward closure; the other is precisely locality,
applied to the covering sieve \(j_J(S)\) and the local covers which
define it.  Applying this observation after every inverse image proves
idempotence.  Thus the fixed points contain top and are closed under
finite intersections and inverse images, so they admit all finite
initial lifts.

For later use, observe the local push identity
\begin{equation}\label{eq:push-saturation}
                 j_J(fR)=j_J(f\,j_J(R)).
\end{equation}
Indeed, if \(g\in j_J(R)\), then \((fg)^*(fR)\) contains the covering
sieve \(g^*R\), and hence \(fg\in j_J(fR)\).  This gives the
nontrivial inclusion; the other follows from \(R\subseteq j_J(R)\).

The right-hand side of \eqref{eq:saturation-final-lift} is the least
saturated sieve containing every \(f_iT_i\), and is therefore the
stated final lift.  Unit presentations follow from idempotence.
Substitution follows by repeated use of \eqref{eq:push-saturation}.
Finally, if \(C\) is saturated, finite-meet preservation of \(j_J\),
\eqref{eq:saturation-base-change}, and the distributivity identity for
sieves give
\[
 \begin{aligned}
 j_J\!\left(\bigcup_i f_iT_i\right)\cap C
 &=j_J\!\left(C\cap\bigcup_i f_iT_i\right)\\
 &=j_J\!\left(\bigcup_i
       f_i(T_i\cap f_i^*C)\right).
 \end{aligned}
\]
This is (D3).  Thus \(\Sat_J\) is a distributivity form.

If \(\C\) has pullbacks, admissible presentations in the covering form
are stable under the Beck--Chevalley formula of Lemma
\ref{lem:sieve-distributivity}, since covering sieves are stable under
inverse image.  For saturated sieves, the same formula followed by
\eqref{eq:saturation-base-change} gives
\[
 h^*j_J\!\left(\bigcup_i f_iT_i\right)
  =j_J\!\left(\bigcup_i f_i'h_i^*T_i\right).
\]
Thus both distributivities satisfy (D4).

Extensivity gives \(S\subseteq j_J(S)\).  If \(T\) is saturated and
\(S\) is locally equivalent to \(T\), then
\(j_J(S)=j_J(T)=T\); hence \(j_J(S)\) is the largest member of the
class of \(S\), and every saturated sieve occurs in this way.
\end{proof}

The preceding construction admits intrinsic converses.  They are useful
because neither description mentions an ambient short exact sequence.

\begin{theorem}\label{thm:characterize-GT-forms}
Let \(F\) be a form over \(\C\) with finite indexed meets.
\begin{enumerate}[label=\textup{(\roman*)}]
\item The form \(F\) is isomorphic to the covering form of a unique
Grothendieck topology if and only if \(\sigma_F\) is fibrewise injective
and its image is upward closed in each sieve fibre and satisfies
locality: whenever \(B\) lies in the image over \(X\) and \(f^*R\) lies
in the image for every \(f\in B\), the sieve \(R\) lies in the image.
\item The form \(F\) is isomorphic to the saturation form of a unique
Grothendieck topology if and only if \(\sigma_F\) is fibrewise injective
and its image is a reflective sub-indexed-poset of \(\Sieve\), with
reflectors
\[
                 j_X:\Sieve_X\longrightarrow\sigma_F(F_X)
\]
which preserve finite meets and commute with inverse images.
\end{enumerate}
In \textup{(i)}, the corresponding distributivity-form
characterization is obtained by requiring the admissible final
presentations to be exactly those maximal sieve presentations whose
sources and apex belong to the image.  In \textup{(ii)}, it is obtained
by requiring every structured sink to be admissible and its indexed
join to have the form
\[
       \bigvee_i{}^{f_i}T_i
         =j_X\!\left(\bigcup_i f_iT_i\right).
\]
These are strong distributivity-form characterizations whenever \(\C\)
has pullbacks.
\end{theorem}

\begin{proof}
The support map preserves top, finite meets, and inverse images.  Hence
the image in \textup{(i)} already contains maximal sieves and is stable
under finite intersection and pullback.  The two additional conditions
are exactly upward closure and Grothendieck locality, so the image is a
Grothendieck topology; injectivity identifies \(F\) with its covering
form.  Conversely, the support of a covering sieve is the sieve itself,
so every covering form has these properties and determines its topology
uniquely.

For \textup{(ii)}, the reflectors amount to an indexed closure operator
\(j\) which is extensive, idempotent, finite-meet preserving, and
commutes with inverse images.  Put
\[
                         J(X)=\{S\mid j_X(S)=\top_X\}.
\]
The standard nucleus argument gives the maximal-sieve, pullback, and
locality axioms.  More explicitly, naturality and the test
\(f\in T\iff f^*T=\top\), for a sieve \(T\), give
\[
 f\in j_X(S)
 \quad\Longleftrightarrow\quad
 j_Y(f^*S)=\top_Y
 \quad\Longleftrightarrow\quad f^*S\in J(Y).
\]
Thus \(j=j_J\), and its fixed points are exactly the saturation form.
The converse follows from Theorem
\ref{thm:topology-gives-two-forms}.  Finally, the two displayed choices
of admissible presentations are precisely the induced and reflected
distributivities constructed there; the strong assertions follow from
the last part of that theorem.
\end{proof}

\subsection{The terminal base}

When the base is the terminal category \(\mathbf1\), a form admitting
finite initial lifts is simply a meet-semilattice \(L\) with top.  A
Beck--Chevalley square is then an identity square, so every
distributivity is automatically strong.  A
structured sink is an ordinary family, and a final presentation is an
ordinary join presentation
\[
                         u=\bigvee_{i\in I}a_i.
\]
Such a presentation is \emph{exact}, or \emph{distributive}, when for
every \(c\in L\) the join on the right below exists and
\begin{equation}\label{eq:exact-join}
              c\wedge u=\bigvee_{i\in I}(c\wedge a_i).
\end{equation}
This is the admissible-join condition of Bruns and Lakser
\cite{BrunsLakser} and the exact-join condition used for semilattice
sites by Ball and Pultr \cite{BallPultr}.

\begin{lemma}\label{lem:maximal-exact-distributivity}
On every meet-semilattice with top, all exact join presentations form
the largest distributivity.
\end{lemma}

\begin{proof}
Unit joins are exact.  If \(u=\bigvee_i a_i\) is exact, then meeting
\eqref{eq:exact-join} once more with \(d\) shows that its restriction
along \(c\) is again exact.  Suppose also that
\(a_i=\bigvee_jb_{ij}\) is exact for each \(i\).  Universal properties
give \(u=\bigvee_{i,j}b_{ij}\), while
\[
 c\wedge u
 =\bigvee_i(c\wedge a_i)
 =\bigvee_{i,j}(c\wedge b_{ij}).
\]
Thus exact joins are closed under substitution and satisfy (D1)--(D3).
Conversely, (D3) says that every join in any distributivity is exact.
\end{proof}

\begin{example}\label{ex:terminal-examples}
Every meet-semilattice with top becomes a distributivity form by
selecting only unit presentations.  A bounded distributive lattice may
instead select all finite joins, including the empty join.  A frame may
select all joins, and a \(\kappa\)-frame may select all joins of
cardinality less than the regular cardinal \(\kappa\).  More generally,
partial frames provide meet-semilattices with specified distributive
joins; one takes the closure of those joins under (D1)--(D3).  See
\cite{Zenk,SchauerteFrith,FrithSchauerteMadden}.
\end{example}

\begin{example}\label{ex:open-set-form}
For a topological space \(X\), the frame \(\mathcal O(X)\), with
intersection as meet and all unions selected, is a distributivity form
over \(\mathbf1\).  A meet-closed basis containing \(X\) gives a
smaller example by selecting exactly those covering families whose
union again belongs to the basis.  In the thin category associated with
the semilattice, the selected families form a subcanonical coverage.
For the maximal exact distributivity this is Stubbe's canonical
topology on a meet-semilattice \cite{Stubbe}.
\end{example}

\section{Extensional presentations}
\label{sec:extensional-presentations}

\subsection{Morphisms, limits, and monomorphisms}

\begin{definition}\label{def:distributivity-morphism}
A \emph{morphism of distributivity forms}
\[
             u:(F,\mathcal P_F)\longrightarrow(G,\mathcal P_G)
\]
over \(\C\) is a natural transformation between the indexed posets
\eqref{eq:indexed-meet-form} whose components preserve finite meets and
which preserves every admissible indexed join:
\[
 u_X\!\left(\bigvee_i{}^{f_i}A_i\right)
       =\bigvee_i{}^{f_i}u_{X_i}(A_i),
\]
with the displayed final presentation admissible in \(G\).  We write
\(\DFrm(\C)\) for the resulting category, and \(\SDFrm(\C)\) for its
full subcategory of strong distributivity forms.
\end{definition}

In lift language, these are precisely the form morphisms preserving
finite initial lifts and the selected final lifts.  Naturality is the
inverse-image equation
\[
                     u_X(f^*B)=f^*u_Y(B).
\]
Because strongness is a property of the chosen presentations rather
than extra operations on arrows, the same definition gives the
morphisms of strong distributivity forms.

Limits are inherited from indexed meet-semilattices: admissibility in a
limit is tested in all coordinates.

\begin{theorem}\label{thm:limits-and-monomorphisms}
Subject to the usual size conventions, both \(\DFrm(\C)\) and
\(\SDFrm(\C)\) have all small limits.  These limits are created by the
forgetful functors to indexed meet-semilattices.  Moreover, a morphism
\(u:F\to G\) is a monomorphism if and only if every component
\[
                       u_X:F_X\longrightarrow G_X
\]
is injective.  In that case every \(u_X\) is an order embedding.
\end{theorem}

\begin{proof}
Let \((F_d)_{d\in D}\) be a small diagram, and first form its limit as
a diagram of indexed meet-semilattices:
\[
                         L_X=\lim_{d\in D}(F_d)_X.
\]
A final presentation in \(L\) is declared admissible precisely when
its image under every limit projection \(L\to F_d\) is admissible.
Both the order and inverse images in \(L\) are calculated
componentwise, so a displayed indexed join is final if and only if all
its projections are final.  Conditions \textup{(D1)--(D3)} consequently
hold componentwise, as does \textup{(D4)} when the objects \(F_d\) are
strong.  If a cone of distributivity morphisms has vertex \(H\), the
induced indexed meet morphism \(H\to L\) preserves every admissible
presentation because all its composites with the limit projections do.
This proves the universal property.  For the empty diagram, the
resulting terminal object has one cluster in each fibre and admits every
small final presentation.  In particular, products are formed
fibrewise with coordinatewise admissibility, and equalizers carry the
distributivity induced from their codomain.

Let
\[
 U:\DFrm(\C)\longrightarrow[\C^{\mathrm{op}},\Meet]
\]
be the forgetful functor.  It has a left adjoint which equips an indexed
meet-semilattice with the least distributivity, obtained by closing the
unit presentations under the required axioms.  Hence \(U\) reflects
monomorphisms.  Monomorphisms of indexed meet-semilattices are computed
componentwise, and a finite-meet morphism is monic precisely when it is
injective.  Conversely, a componentwise injective natural
transformation is plainly left cancellable.  The least distributivity
is strong, so the same argument applies to \(\SDFrm(\C)\).

Finally, if \(u_X\) is injective and \(u_X(A)\leqslant u_X(B)\), then
\[
 u_X(A)=u_X(A)\wedge u_X(B)=u_X(A\wedge B),
\]
whence \(A=A\wedge B\), and therefore \(A\leqslant B\).
\end{proof}

\subsection{The constant-top ideals}

A fibrewise constant-top assignment need not be a morphism.  Applied to
a presentation indexed by \((f_i:X_i\to X)_i\), preservation would
require
\[
                    \top_X=\nabla_i(\top_{X_i},f_i)
\]
to be admissible in the codomain.  For sieves, the unary final lift of
\(\top_Y\) along \(f:Y\to X\) is the principal sieve generated by
\(f\), and is not generally maximal.  We therefore retain only the
constant-top assignments which genuinely preserve the structure.

\begin{definition}\label{def:null-ideal}
For distributivity forms \(F,G\), let \(\Null(F,G)\) consist of the
morphisms \(n:F\to G\) such that
\[
                        n_X(A)=\top_X^G
\]
for every \(X\) and \(A\in F_X\).  Such arrows are called \emph{null}.
The same definition, restricted to strong objects, gives an ideal
\(\Null_{\mathrm s}\) in \(\SDFrm(\C)\).
\end{definition}

Each of these classes is a two-sided ideal: if \(n\) is null, then
\(un\) and \(nv\), whenever defined, are null because every morphism
preserves top.  It is important that \(\Null(F,G)\) may be empty; the
categories are not in general pointed.

Recall that an ideal is \emph{closed} if every null arrow factors
through a null object, that is, an object whose identity is null
\cite{GrandisRelative,GrandisBook}.

\begin{lemma}\label{lem:null-ideal-closed}
The ideals \(\Null\) and \(\Null_{\mathrm s}\) are closed.
\end{lemma}

\begin{proof}
Let \(n:F\to G\) be null.  Take the form \(Z_n\) having one cluster in
each fibre.  Select in \(Z_n\) the least distributivity containing the
indexing shapes of the images of all admissible presentations of \(F\).
Then the constant map \(F\to Z_n\) is a morphism.  Because \(n\) itself
preserves the admissible presentations of \(F\), the constant map
\(Z_n\to G\) preserves the generators and hence their closure.  Thus
\(n\) factors through \(Z_n\).  The identity of \(Z_n\) is constant top,
so \(Z_n\) is null.  In the strong case take the least strong closure of
the same shapes; the identical argument uses (D4) as one further closure
rule.
\end{proof}

For an ideal \(\Null\), a \emph{relative kernel} of \(u:F\to G\) is a
morphism \(k:K\to F\) such that \(uk\) is null and through which every
\(h:H\to F\) with \(uh\) null factors uniquely.  A relative cokernel is
defined dually.  We henceforth omit ``relative'' when the ideal is
clear.

\subsection{Kernels, cokernels, and congruences}

For a morphism \(u:F\to G\), put
\begin{equation}\label{eq:kernel-fibres}
              K_X=\{A\in F_X\mid u_X(A)=\top_X^G\}.
\end{equation}
Give \(K\) the induced distributivity: a presentation of \(F\) is
admitted in \(K\) when all its sources and its apex lie in \(K\).

\begin{theorem}\label{thm:kernels}
The fibres in \eqref{eq:kernel-fibres}, with their induced
distributivity, form the \(\Null\)-kernel of \(u\).  If \(F\) and \(G\)
are strong, the induced distributivity on \(K\) is strong and the same
inclusion is the \(\Null_{\mathrm s}\)-kernel.  Thus both categories have
all ideal-relative kernels.
\end{theorem}

\begin{proof}
The fibres contain top and are closed under finite meets and inverse
images.  An admissible final lift of \(F\) whose data lie in \(K\)
remains final there, while (D1)--(D3) stay internal.  The same is true of
(D4), since the kernel is stable under inverse image.  The composite
\(K\to F\xrightarrow{u}G\) is null.  If \(h:H\to F\) has \(uh\) null,
every value of \(h\) lies in \(K\), and the induced definition makes the
unique set-theoretic factor \(H\to K\) a morphism.  This is the required
universal property in either category.
\end{proof}

Every kernel fibre is a filter in the ambient fibre.  What distinguishes
kernels from arbitrary pullback-stable fibrewise filters is their
compatibility with the admissible indexed joins.

An \emph{indexed meet congruence} \(\theta\) on \(F\) is a family of
meet congruences \(\theta_X\) stable under inverse image.

\begin{definition}\label{def:distributivity-congruence}
An indexed meet congruence \(\theta\) on a distributivity form is a
\emph{distributivity congruence} when, for admissible presentations with
the same indexing arrows,
\[
 \begin{gathered}
 U=\nabla_i(A_i,f_i),\qquad
 V=\nabla_i(B_i,f_i),\\
 A_i\mathrel\theta B_i\ \text{for every }i
 \quad\Longrightarrow\quad U\mathrel\theta V.
 \end{gathered}
\]
\end{definition}

If \(q:F\to F/\theta\) is the fibrewise quotient by such a congruence,
the image of every admissible presentation is again a genuine final
lift.  Indeed, if \([A_i]\leqslant f_i^*[C]\), then
\(A_i\mathrel\theta A_i\wedge f_i^*C\).  Compatibility with the original
presentation and its (D3)-restriction yields
\(U\mathrel\theta U\wedge C\), so \([U]\leqslant[C]\).  This is the
point at which distributivity is essential.

Intersections of distributivity congruences are distributivity
congruences.  For \(u:F\to G\), let \(\theta_u\) be the least such
congruence on \(G\) satisfying
\begin{equation}\label{eq:cokernel-generators}
                         u(A)\mathrel{\theta_u}\top
\end{equation}
for every cluster \(A\) of \(F\).  Equip \(G/\theta_u\) with the
distributivity generated by the images of the admissible presentations
of \(G\).  In the strong case take their strong closure under
(D1)--(D4).

\begin{theorem}\label{thm:cokernels}
The quotient map
\[
                 q_u:G\longrightarrow G/\theta_u
\]
is the \(\Null\)-cokernel of \(u\), and in the strong category it is the
\(\Null_{\mathrm s}\)-cokernel.  Consequently
\((\DFrm(\C),\Null)\) and \((\SDFrm(\C),\Null_{\mathrm s})\) are
semiexact categories in the sense of Grandis.
\end{theorem}

\begin{proof}
The preceding argument shows that images of admissible presentations
are genuine final lifts in the quotient.  Closing them under (D1)--(D3)
therefore defines a distributivity.  In the strong case, the image of a
Beck--Chevalley equality is again that equality because \(q_u\) commutes
with inverse image, so closing under (D4) defines a strong
distributivity.  The map \(q_u\) preserves all selected structure and
\(q_uu\) is null.

If \(r:G\to H\) has \(ru\) null, equality under \(r\) is a
distributivity congruence containing the generators
\eqref{eq:cokernel-generators}.  Thus \(r\) factors uniquely through
\(q_u\), and the factor preserves the generated distributivity.  The
same proof works in the strong closure.  Hence \(q_u\) has the stated
universal property.  Lemma~\ref{lem:null-ideal-closed}, together with
the existence of all relative kernels and cokernels, is precisely the
semiexactness assertion.
\end{proof}

For a pullback-stable fibrewise filter \(K\subseteq F\), let
\(\theta_K\) be the least distributivity congruence sending every member
of \(K\) to top, and put
\[
       \operatorname{ncl}_{\mathcal P}(K)
          =\{A\in F\mid A\mathrel{\theta_K}\top\}.
\]

\begin{lemma}\label{lem:kernel-criterion}
With the induced distributivity, \(K\hookrightarrow F\) is a relative
kernel if and only if
\[
              \operatorname{ncl}_{\mathcal P}(K)=K.
\]
In that case its canonical cokernel is \(F\to F/\theta_K\).  The same
criterion holds for strong distributivity forms.
\end{lemma}

\begin{proof}
If the equality holds, \(K\) is exactly the kernel of its cokernel by
Theorem~\ref{thm:cokernels}.  Conversely, if \(K=\ker u\), equality
under \(u\) is a distributivity congruence killing \(K\), so minimality
gives \(\operatorname{ncl}_{\mathcal P}(K)\subseteq K\); the reverse
inclusion is built into the generators.  The strong proof is identical.
\end{proof}

\subsection{Normal factorization and stability}

A morphism \(m:F\to G\) is \emph{top-reflecting} when, for every
\(X\in\C\) and \(A\in F_X\),
\[
       m_X(A)=\top_X^G\quad\Longrightarrow\quad A=\top_X^F.
\]
Equivalently, its relative kernel is a null object.  Let \(\mathcal E\)
be the class of relative cokernels and \(\mathcal M\) the class of
top-reflecting morphisms.

\begin{theorem}\label{thm:cokernel-factorization-system}
The pair
\[
             (\mathcal E,\mathcal M)
       =(\textup{relative cokernels},\textup{top-reflecting morphisms})
\]
is an orthogonal factorization system on \(\DFrm(\C)\).  The analogous
pair is an orthogonal factorization system on \(\SDFrm(\C)\).  Every
monomorphism is top-reflecting, but the converse need not hold.
\end{theorem}

\begin{proof}
Let \(u:F\to G\), put \(K=\ker u\), and factor \(u\) through the
cokernel of its kernel:
\[
 F\xrightarrow{q_K}F/\theta_K
       \xrightarrow{\overline u}G.
\]
Equality under \(u\) is a distributivity congruence killing \(K\), so
the second map exists.  If
\(\overline u([A])=\top\), then \(u(A)=\top\), hence \(A\in K\) and
\([A]=\top\).  Thus \(\overline u\) is top-reflecting, and every
morphism has an \(\mathcal E\)--\(\mathcal M\) factorization.

Let \(e:A\to B\) be the cokernel of \(k:K\to A\), and consider a
commutative square
\[
\begin{array}{ccc}
A&\xrightarrow{a}&X\\
{\scriptstyle e}\downarrow&&\downarrow{\scriptstyle m}\\
B&\xrightarrow{b}&Y
\end{array}
\]
with \(m\) top-reflecting.  Since \(ek\) is null,
\(mak=bek\) is null.  Top-reflection gives that \(ak\) is null, so the
cokernel property gives a unique \(d:B\to X\) with \(de=a\).  Every
relative cokernel is epic; hence \(mde=ma=be\) implies \(md=b\).
This proves \(\mathcal E\perp\mathcal M\).

Conversely, suppose that \(m:X\to Y\) is right orthogonal to every
relative cokernel, and factor it as
\[
             X\xrightarrow{q}Q\xrightarrow{\overline m}Y
\]
by the construction above.  Applying the lifting property to the square
with top map \(1_X\) and bottom map \(\overline m\) gives a map
\(d:Q\to X\) such that \(dq=1_X\).  Since \(q\) is epic, also
\(qd=1_Q\).  Thus \(q\) is an isomorphism, and \(m\) is
top-reflecting.  It follows that \(\mathcal M=\mathcal E^\perp\).

Now suppose that \(\ell:A\to B\) lies in
\({}^\perp\mathcal M\), and factor it as
\[
                    A\xrightarrow{e}Q\xrightarrow{m}B.
\]
The lifting property gives \(d:B\to Q\) with \(d\ell=e\) and
\(md=1_B\).  If \(e\) is the cokernel of \(k:K\to A\), then
\(\ell k\) is null.  Given \(t:A\to T\) with \(tk\) null, write
uniquely \(t=he\).  Then \(t=(hd)\ell\).  If also \(t=s\ell\), the
epimorphism \(e\) gives \(sm=h\), and hence \(s=smd=hd\).  Thus
\(\ell\) is itself the cokernel of \(k\), and
\(\mathcal E={}^\perp\mathcal M\).

Every monomorphism reflects top by Theorem
\ref{thm:limits-and-monomorphisms}, since it preserves top and is
fibrewise injective.  The inclusion is strict already over the terminal
base.  Give the three-element and two-element chains their least
distributivities.  The meet morphism
\[
 \{0<a<1\}\longrightarrow\{0<1\},\qquad
 0,a\longmapsto0,\quad 1\longmapsto1,
\]
reflects top but is not injective.
\end{proof}

The categories are semiexact but need not be homological.  The precise
composition behaviour is asymmetric.

\begin{theorem}\label{thm:composition-normal-maps}
Kernels need not be closed under composition in either
\((\DFrm(\C),\Null)\) or \((\SDFrm(\C),\Null_{\mathrm s})\).  Cokernels,
by contrast, are closed under composition in both categories.
\end{theorem}

\begin{proof}
For the first assertion it suffices to work over the terminal base, where
strongness is automatic.  Put \(S=\{p\}\amalg\mathbb N\) and
\(M=\mathcal P(S)\), with intersection as meet.  For
\(I=\{*\}\amalg\mathbb N\), set
\[
 A_*=\{p\},\quad A_n=\{n\},\qquad
 B_*=\{p\},\quad B_n=\varnothing.
\]
Give \(M\) the distributivity generated by units and the two exact join
presentations
\[
                  S=\bigvee_{i\in I}A_i,\qquad
              \{p\}=\bigvee_{i\in I}B_i.
\]
Point evaluation \(u:M\to\mathbf2\), defined by
\(u(C)=1\) exactly when \(p\in C\), is a distributivity morphism.  Hence
\[
                         L=\ker u=\{C\subseteq S\mid p\in C\}
\]
is a kernel.  Every nonunit presentation generated in \(M\) has at
least one source not containing \(p\), a property preserved by
substitution and intersection.  The induced distributivity on \(L\)
therefore contains only units.

Let
\[
 K=\{C\in L\mid\mathbb N\setminus C\text{ is finite}\}.
\]
The characteristic map \(v:L\to\mathbf2\) preserves finite meets and
the unit distributivity, and \(K=\ker v\).  Thus
\(K\hookrightarrow L\) and \(L\hookrightarrow M\) are kernels.  Their
composite is not.  Indeed, if a morphism \(w:M\to H\) kills \(K\), then
for \(k_n=S\setminus\{n\}\in K\),
\[
 w(\{n\})=w(\{n\})\wedge w(k_n)=w(\varnothing).
\]
The two selected presentations consequently have componentwise equal
images, whence \(w(S)=w(\{p\})\).  Thus \(w(\{p\})=\top\), although
\(\{p\}\notin K\).  No morphism has precisely \(K\) as its kernel.

For cokernels, every cokernel is isomorphic to a fibrewise-surjective
canonical quotient and is the cokernel of its own kernel.  Indeed, if
its congruence is \(\theta\) and \(K\) is the top class, then the least
congruence killing \(K\) is \(\theta\): each inclusion follows from one
of the two defining minimalities.  Let \(q:F\to G\) and \(r:G\to H\)
be successive cokernels, put \(K=\ker(rq)\), and let \(p:F\to P\) be
the cokernel of \(K\).  Since \(\ker q\subseteq K\), write \(p=bq\).
Surjectivity of \(q\) shows that \(b\) kills \(\ker r\), so \(b=cr\).
Conversely, write \(rq=ap\).  Epimorphic cancellation gives
\(ca=1_P\) and \(ac=1_H\).  Hence \(rq\) is isomorphic to \(p\), and is
a cokernel.  The argument is unchanged in the strong category.
\end{proof}

The preceding factorization system does not imply stability of its left
class under pullback.  The obstruction can be seen without any
nontrivial reindexing in the base.

\begin{theorem}\label{thm:cokernels-not-pullback-stable}
Relative cokernels are not stable under pullback in either
\(\DFrm(\C)\) or \(\SDFrm(\C)\).  The failure occurs already when
\(\C=\mathbf1\), and even for pullback along a top-reflecting morphism.
\end{theorem}

\begin{proof}
Put
\[
 S=\{r,p\}\amalg\mathbb N,\qquad
 T=\{p\}\amalg\mathbb N,\qquad
 I=\{*\}\amalg\mathbb N,
\]
and let \(M=\mathcal P(S)\), with intersection as meet.  Set
\[
 A_* = B_*=\{p\},\qquad A_n=\{n\},\qquad B_n=\varnothing .
\]
Equip \(M\) with the distributivity generated by the two exact
presentations
\begin{equation}\label{eq:pullback-counterexample-presentations}
                    T=\bigvee_{i\in I}A_i,
             \qquad \{p\}=\bigvee_{i\in I}B_i .
\end{equation}
Let
\[
 D=\{d\subseteq S\mid r,p\in d\text{ and }
                     \mathbb N\setminus d\text{ is finite}\}.
\]
This filter is normal.  Indeed, its characteristic meet morphism
\(\chi_D:M\to\mathbf2\), defined by \(\chi_D(C)=1\) precisely when
\(C\in D\), becomes a distributivity morphism if \(\mathbf2\) is
equipped with the distributivity generated by
\[
                          0=\bigvee_{i\in I}0.
\]
Both presentations in
\eqref{eq:pullback-counterexample-presentations} map to this
presentation.  Thus \(D=\ker\chi_D\).

Let
\[
                    q:M\longrightarrow Q=M/\theta_D
\]
be the cokernel of \(D\hookrightarrow M\).  For each \(n\), the member
\(d_n=S\setminus\{n\}\) lies in \(D\).  Since
\(d_n\mathrel{\theta_D}S\), meeting with \(\{n\}\) gives
\[
                     \{n\}\mathrel{\theta_D}\varnothing .
\]
Compatibility with the two selected presentations therefore yields
\[
                     q(T)=q(\{p\})=:a.
\]
Moreover \(a\neq\top\), because the top class of \(\theta_D\) is
exactly \(D\), whereas neither \(T\) nor \(\{p\}\) belongs to \(D\).

Give the two-element meet-semilattice \(H=\{0<1\}\) its least
distributivity and define
\[
                   j:H\longrightarrow Q,\qquad
                   j(0)=a,\quad j(1)=\top .
\]
This is a top-reflecting distributivity morphism.  Form the pullback
\[
\begin{array}{ccc}
P=M\times_QH&\xrightarrow{\pi_M}&M\\
{\scriptstyle\pi_H}\downarrow&&\downarrow{\scriptstyle q}\\
H&\xrightarrow{j}&Q .
\end{array}
\]
By Theorem~\ref{thm:limits-and-monomorphisms}, the distributivity of
\(P\) is detected coordinatewise.  Since \(H\) has only unit
presentations, so does \(P\).  The two elements
\[
                         x=(T,0),\qquad y=(\{p\},0)
\]
belong to \(P\) and satisfy \(\pi_H(x)=\pi_H(y)\).

The kernel of \(\pi_H\) is the filter
\[
                   \widetilde D=\{(d,1)\mid d\in D\}.
\]
For a meet-semilattice with only unit presentations, the least
congruence killing a filter is the common-witness congruence
\[
 z\equiv_{\widetilde D}z'
 \quad\Longleftrightarrow\quad
 z\wedge k=z'\wedge k
 \quad\text{for some }k\in\widetilde D .
\]
If this relation identified \(x\) and \(y\), some \(d\in D\) would
satisfy
\[
                         T\cap d=\{p\}\cap d=\{p\}.
\]
This is impossible, since \(d\cap\mathbb N\) is cofinite.  Hence the
cokernel of \(\ker\pi_H\) does not identify \(x\) and \(y\), whereas
\(\pi_H\) does.  Every cokernel is the cokernel of its own kernel, as
shown in Theorem~\ref{thm:composition-normal-maps}; therefore
\(\pi_H\) is not a cokernel.

Over the terminal base every distributivity is strong.  The same
example proves the assertion in both categories.
\end{proof}

\subsection{Extensional and strong extensional presentations}

A composable pair
\begin{equation}\label{eq:abstract-short-exact}
 K\xrightarrow{k}F\xrightarrow{q}Q
\end{equation}
is \emph{\(\Null\)-short exact} when \(k\) is a \(\Null\)-kernel of
\(q\) and \(q\) is a \(\Null\)-cokernel of \(k\).  The analogous
definition uses \(\Null_{\mathrm s}\) in the strong category.

\begin{definition}\label{def:extensional-presentation}
An \emph{extensional presentation} of a distributivity form \(F\) is an
isomorphism class, fixing \(F\), of \(\Null\)-short exact sequences
\eqref{eq:abstract-short-exact} having middle term \(F\).  A
\emph{strong extensional presentation} is defined in the same way in
\((\SDFrm(\C),\Null_{\mathrm s})\).
\end{definition}

The adjective ``extensional'' records that the right-hand term
identifies clusters exactly as forced by those declared indistinguishable
from top.  An extensional presentation is a short-exact datum and should
not be confused with an admissible final presentation, which is a
selected indexed join.

By Theorems~\ref{thm:kernels} and \ref{thm:cokernels}, an extensional
presentation may equivalently be specified by a pullback-stable
fibrewise filter \(K\subseteq F\) satisfying the normal-closure
condition of Lemma~\ref{lem:kernel-criterion}.  Its canonical
representative is
\[
 K\longrightarrow F\longrightarrow F/\theta_K.
\]

\begin{theorem}\label{thm:ordinary-strong-presentations}
For a distributivity form \(F\), the following are equivalent.
\begin{enumerate}[label=\textup{(\roman*)}]
\item The distributivity of \(F\) is strong.
\item Every extensional presentation of \(F\) is canonically a strong
extensional presentation.
\item The canonical extensional presentation
\[
       \{\top\}\longrightarrow F\xrightarrow{1_F}F
\]
is a strong extensional presentation.
\end{enumerate}
Consequently ordinary and strong extensional presentations coincide
precisely on strong distributivity forms.
\end{theorem}

\begin{proof}
The inclusion \(\SDFrm(\C)\hookrightarrow\DFrm(\C)\) creates the
relative kernels and cokernels of Theorems~\ref{thm:kernels} and
\ref{thm:cokernels}: induced subforms inherit (D4), while the ordinary
distributivity generated on a canonical quotient is already stable
under (D4).  Indeed, base change carries an image generator to an image
generator, and it commutes with units, substitution, and Frobenius
restriction.  Thus an
extensional presentation of a strong middle term is isomorphic to its
canonical representative, all three terms of which are strong.  This
proves \textup{(i)}\(\Rightarrow\)\textup{(ii)}; the next implication is
immediate.  Condition \textup{(iii)} can hold in the strong category
only if its middle object \(F\) is strong, giving the converse.
\end{proof}

\subsection{Closure operators and Hausdorff objects}

Many familiar extensional presentations are more naturally displayed
by a closure operator.  Call a right adjoint \(s\) to \(q:F\to Q\)
\emph{indexed} when every \(q_X\) has right adjoint \(s_X\) and
\[
                         f^*s_Y=s_Xf^*
\]
for all \(f:X\to Y\).

\begin{theorem}\label{thm:adjoint-closure}
Suppose that the quotient \(q:F\to Q\) in an extensional presentation
of \(F\) has an indexed right adjoint \(s\).  Then
\[
                          c_X=s_Xq_X
\]
is an extensive, monotone, idempotent, finite-meet-preserving closure
operator on every fibre, and it commutes with inverse image.  Moreover,
\[
 K_X=\{A\in F_X\mid c_X(A)=\top_X\},
\]
and the extensional presentation is isomorphic to
\begin{equation}\label{eq:closure-extensional-presentation}
 K\longrightarrow F
  \xrightarrow{\,A\mapsto c(A)\,}\operatorname{Fix}(c),
\end{equation}
where the fixed-point form carries the distributivity transported from
\(Q\).
\end{theorem}

\begin{proof}
A cokernel is isomorphic to the fibrewise surjective quotient in
Theorem~\ref{thm:cokernels}; we may therefore assume that every \(q_X\)
is surjective.  For an adjunction of posets \(q_X\dashv s_X\),
surjectivity gives \(q_Xs_X=1\).  Consequently \(s_Xq_X\) is extensive
and idempotent.  It is monotone, and it preserves finite meets because
\(q_X\) does and the right adjoint \(s_X\) preserves all existing meets.
Naturality of \(q\) and the indexed condition on \(s\) give
\[
 c_Xf^*=s_Xq_Xf^*=s_Xf^*q_Y=f^*s_Yq_Y=f^*c_Y.
\]
Since \(s_X\) preserves top and \(q_Xs_X=1\), one has
\(q_X(A)=\top\) if and only if \(c_X(A)=\top\).  This identifies the
kernel.  Finally, \(s_X\) identifies \(Q_X\) order-isomorphically with
the fixed points of \(c_X\), and these isomorphisms commute with inverse
image.  Transporting the distributivity gives
\eqref{eq:closure-extensional-presentation}.
\end{proof}

Conversely, an indexed productive closure operator \(c\) gives an
extensional presentation whenever \(F\to\operatorname{Fix}(c)\)
preserves the distributivity and is the cokernel of its top class.  We
call such a closure operator \emph{exact for the distributivity}.  The
qualifier is necessary: over the terminal base an arbitrary frame
nucleus need not be determined by its top class.  Theorem
\ref{thm:frame-topologies} identifies the nuclei which are.

Let \(\mathbb E\) be an extensional presentation of \(F\), with
canonical representative
\[
                  K\longrightarrow F\longrightarrow F/\theta_K.
\]
The kernel is the part of the datum which plays the role of a system of
neighbourhoods.

\begin{definition}\label{def:hausdorff-object}
Suppose that the fibre \(F_X\) has a least cluster \(\bot_X\).  The
object \(X\) is \emph{Hausdorff for \(\mathbb E\)} when \(\bot_X\) is
the greatest lower bound of \(K_X\).  Equivalently,
\[
 C\leqslant A\text{ for every }A\in K_X
 \quad\Longrightarrow\quad C=\bot_X.
\]
When the meet exists, put
\[
                  \rho_{\mathbb E}(X)=\bigwedge_{A\in K_X}A
\]
and call it the \emph{separation radical}.  The extensional presentation
is \emph{Hausdorff} when every object of its base is Hausdorff.
\end{definition}

The first formulation does not assume arbitrary meets in the fibres and
does not refer to a chosen representative of the quotient.

\begin{lemma}\label{lem:hausdorff-intrinsic}
Hausdorffness depends only on the extensional presentation, not on its
representative sequence.  If \(K_X\) is the principal filter generated
by \(R_X\), then \(\rho_{\mathbb E}(X)=R_X\); hence \(X\) is Hausdorff
if and only if \(R_X=\bot_X\).
\end{lemma}

\begin{proof}
An isomorphism of representatives fixes \(F\), so its kernel terms
determine the same subform and have the same lower bounds.  If
\(K_X=\{A\mid R_X\leqslant A\}\), then \(R_X\) belongs to \(K_X\) and
is below all its members, and is therefore their greatest lower bound.
\end{proof}

In a non-Archimedean topological group this radical is the intersection
of the open subgroups, and in a linearly topologized ring it is the
intersection of the open ideals.  Thus the definition recovers ordinary
Hausdorffness in the algebraic applications below.  It is a relative
separation notion and should not be confused with separation for a
Lawvere--Tierney topology, which is expressed by closedness of a
diagonal.  Without additional idempotence assumptions, the radical need
not belong to the kernel filter and quotienting by it need not produce a
Hausdorff reflection.

\section{Applications}
\label{sec:applications}

We now apply extensional presentations to the principal geometric,
order-theoretic, and algebraic forms.  The first theorem is the
motivation for the definition.

\subsection{Grothendieck topologies and the sieve form}

We now prove the central characterization.  The form \(\Sieve\) is
always understood to carry the maximal distributivity of
Lemma~\ref{lem:sieve-distributivity}.

\begin{theorem}\label{thm:main}
For every small category \(\C\), Grothendieck topologies on \(\C\) are
naturally in bijection with extensional presentations of the maximally
distributive form of sieves.  The representative associated with
\(J\) is
\begin{equation}\label{eq:main-short-exact}
 J\longrightarrow\Sieve
 \xrightarrow{\,q_J\,}\Sat_J,
 \qquad q_J(S)=j_J(S).
\end{equation}
Thus the kernel consists of the covering sieves, while the cokernel
identifies two sieves precisely when they cover the same arrows.  If
\(\C\) has pullbacks, the same correspondence is a bijection with strong
extensional presentations of the strong sieve form.
\end{theorem}

\begin{proof}
Start with an extensional presentation of \(\Sieve\), represented by
\[
 K\longrightarrow\Sieve\xrightarrow{q}Q.
\]
Using Theorem~\ref{thm:kernels}, we may identify \(K\) with the top
class of \(q\).  It contains every maximal sieve, is stable under
inverse image, and is a fibrewise filter, so it is upward closed and
closed under finite intersections.  It remains to prove Grothendieck
locality.

If \(q\) is null, then \(K=\Sieve\), which is the chaotic
Grothendieck topology.  We may therefore suppose that \(q\) is a
non-null morphism of distributivity forms.

Let \(B\in K_X\), let \(R\) be a sieve on \(X\), and suppose that
\(f^*R\in K_{\operatorname{dom}f}\) for every \(f\in B\).  In the
maximally distributive form of sieves there are admissible final
presentations
\[
 \begin{aligned}
 S:=B\cap R
   &=\nabla_{f\in B}(f^*R,f),\\
 B&=\nabla_{f\in B}(\top_{\operatorname{dom}f},f).
 \end{aligned}
\]
Indeed, both equalities follow directly from union formula
\eqref{eq:sieve-final-lift}.  The corresponding source clusters have
the same image under \(q\), since each \(f^*R\) lies in the kernel.
Preservation of the two presentations therefore gives
\[
                         q(B\cap R)=q(B)=\top.
\]
Hence \(B\cap R\in K_X\), and upward closure gives \(R\in K_X\).
This is the transitivity axiom, so the fibres of \(K\) are precisely the
covering sieves of a Grothendieck topology.

Conversely, let \(J\) be a Grothendieck topology.  Section
\ref{sec:distributivity-forms} constructed the covering form \(J\), the
saturation form \(\Sat_J\), and the operators \(j_J\).  The map
\(q_J(S)=j_J(S)\) preserves top, finite meets, and inverse images.  It
also preserves every admissible final lift: by
\eqref{eq:push-saturation},
\[
 \begin{aligned}
 q_J\!\left(\nabla_i(S_i,f_i)\right)
 &=j_J\!\left(\bigcup_i f_iS_i\right)\\
 &=j_J\!\left(\bigcup_i f_i j_J(S_i)\right)
  =\nabla_i(q_J(S_i),f_i),
 \end{aligned}
\]
and the final presentation on the right is admissible in \(\Sat_J\).
Thus \(q_J\) is a morphism of distributivity forms.  Moreover,
\[
 j_J(S)=\top_X
 \quad\Longleftrightarrow\quad
 1_X\in j_J(S)
 \quad\Longleftrightarrow\quad
 S\in J(X),
\]
so its kernel is the covering form \(J\), with the induced
distributivity.

We prove directly that \(q_J\) is the cokernel of this kernel.  Let
\(r:\Sieve\to H\) be a morphism whose restriction to the covering form
is null.  Fix a sieve \(S\) on \(X\), and
put \(B=j_J(S)\).  For every \(f\in B\), the sieve \(f^*S\) covers by
definition.  Since \(S\subseteq B\), the two admissible presentations
\[
 S=B\cap S=\nabla_{f\in B}(f^*S,f),
 \qquad
 B=\nabla_{f\in B}(\top,f)
\]
have source clusters with equal images under \(r\).  Hence
\[
                              r(S)=r(j_J(S)).
\]
If \(j_J(S)=j_J(T)\), it follows that \(r(S)=r(T)\).  Thus \(r\)
factors uniquely through the surjection \(q_J\).  The factor preserves
finite initial lifts, and it preserves the quotient distributivity
because the latter is generated by the images of the presentations of
\(\Sieve\).  This proves the relative cokernel universal property, and
\eqref{eq:main-short-exact} is \(\Null\)-short exact.

The two assignments are inverse.  Beginning with \(J\), the kernel of
\(q_J\) recovers exactly \(J\).  Beginning with an extensional
presentation of \(\Sieve\), its kernel gives the Grothendieck topology
just constructed;
both its original quotient and \(q_J\) are cokernels of the same
inclusion \(J\to\Sieve\), and are therefore uniquely isomorphic under
\(\Sieve\).  This proves naturality.  When \(\C\) has pullbacks, the
maximal sieve distributivity and the two forms associated with \(J\)
are strong by Lemma~\ref{lem:sieve-distributivity} and Theorem
\ref{thm:topology-gives-two-forms}; Theorem
\ref{thm:ordinary-strong-presentations} gives the final assertion.
\end{proof}

The quotient in \eqref{eq:main-short-exact} has the indexed right
adjoint which includes saturated sieves into all sieves.  Hence
Theorem~\ref{thm:adjoint-closure} recovers the usual nucleus \(j_J\)
from the extensional presentation.  The proof above also shows directly
that local equivalence is the least distributivity congruence forcing
every covering sieve to be top.

\subsection{Variations and singleton-generated topologies}

Assume in this subsection that \(\C\) has pullbacks.  For an arrow
\(f:Y\to X\), write
\[
                  \prin f=\{fg\mid \operatorname{cod}g=Y\}
\]
for its principal sieve.  Two arrows generate the same principal sieve
precisely when each factors through the other.  Pullbacks ensure that
inverse images and finite intersections of principal sieves are
principal.  They therefore form a subform
\(\Var\hookrightarrow\Sieve\), called the \emph{form of variations}.
Give it the strong unary distributivity generated by
\[
                  \prin{fg}=\nabla(\prin g,f)
\]
for composable arrows \(g,f\), together with (D1)--(D4).  In
particular, \(\prin f=\nabla(\top,f)\).

A Grothendieck topology is \emph{singleton-generated} when every
covering sieve contains a covering principal sieve.  It is useful to
encode such a topology by its covering arrows.  Call a class \(W\) of
arrows a \emph{saturated singleton coverage} when:
\begin{enumerate}[label=\textup{(S\arabic*)}]
\item identities belong to \(W\);
\item \(W\) is stable under pullback;
\item \(W\) is closed under composition;
\item \(fg\in W\) implies \(f\in W\).
\end{enumerate}
Its associated family of sieves is
\[
 J_W(X)=\{S\in\Sieve_X\mid
   \prin w\subseteq S\text{ for some }w\in W
   \text{ with codomain }X\}.
\]

\begin{theorem}\label{thm:singleton-topologies}
For a category \(\C\) with pullbacks, singleton-generated
Grothendieck topologies on \(\C\) are naturally in bijection with
extensional presentations, equivalently strong extensional
presentations, of the unary distributivity form \(\Var\).  If \(W\) is
the corresponding saturated singleton coverage, a representative is
\begin{equation}\label{eq:variation-short-exact}
 E_W\longrightarrow\Var
 \xrightarrow{\,q_W\,}Q_W,
\end{equation}
where
\[
 E_W=\{\prin w\mid w\in W\},
 \qquad
 q_W(\prin f)=j_{J_W}(\prin f),
\]
and \(Q_W\) consists of the saturated principal sieves
\(j_{J_W}(\prin f)\), with the quotient unary distributivity.
\end{theorem}

\begin{proof}
First let \(W\) be a saturated singleton coverage.  The families
\(J_W(X)\) form a Grothendieck topology.  Maximal sieves cover by (S1),
and stability follows from (S2).  For transitivity, suppose that a
covering sieve \(B\) contains \(\prin w\), where \(w\in W\), and that
every pullback of \(R\) along an arrow of \(B\) covers.  In particular,
\(w^*R\) contains \(\prin u\) for some \(u\in W\).  Then \(wu\in R\)
and \(wu\in W\) by (S3), so \(R\) covers.  Moreover,
\begin{equation}\label{eq:principal-cover-test}
                  \prin f\in J_W(X)\quad\Longleftrightarrow\quad f\in W.
\end{equation}
Indeed, if \(\prin w\subseteq\prin f\), then \(w=fv\) for some \(v\),
and (S4) gives \(f\in W\); the converse is immediate.

The clusters \(E_W\) contain top and are stable under inverse image.
They are upward closed among principal sieves by (S4).  They are closed
under meets: if \(f,g\in W\) have a common codomain and \(p\) is the
pullback of \(g\) along \(f\), then \(p\in W\), \(fp\in W\), and
\(\prin f\cap\prin g=\prin{fp}\).  Hence \(E_W\) is a
pullback-stable fibrewise filter in \(\Var\).

The nucleus of \(J_W\) sends principal sieves to a family closed under
top, finite meets, and inverse image; call the resulting form \(Q_W\).
The map \(q_W\) preserves these finite initial lifts.  The push identity
\eqref{eq:push-saturation} gives the quotient unary presentations
\[
 j_{J_W}(\prin{fg})
   =\nabla\bigl(j_{J_W}(\prin g),f\bigr),
\]
so \(q_W\) is a morphism of distributivity forms.  By
\eqref{eq:principal-cover-test}, its kernel is \(E_W\).

To prove the cokernel property, let \(r:\Var\to H\) kill \(E_W\).
If \(f\in j_{J_W}(\prin g)\), then \(f^*\prin g\) contains a covering
principal sieve.  Thus there are arrows \(u\in W\) and \(v\) with
\(fu=gv\).  Since \(r(\prin u)=\top\), preservation of the unary
presentations along \(f\) gives
\[
 r(\prin f)=r(\prin{fu})=r(\prin{gv})\leqslant r(\prin g).
\]
If two principal sieves have the same saturation, this argument in both
directions makes their images under \(r\) equal.  Hence \(r\) factors
uniquely through \(q_W\), and the factor preserves the quotient
distributivity.  Thus \eqref{eq:variation-short-exact} is an
extensional presentation of \(\Var\).

Conversely, let an extensional presentation
\[
 E\longrightarrow\Var\xrightarrow{q}Q
\]
be given, and set \(W=\{f\mid\prin f\in E\}\).  The kernel
properties give (S1), (S2), and (S4).  If \(f,g\in W\) are composable,
the conclusion is automatic when \(q\) is null, since then \(E=\Var\)
and \(W\) contains every arrow.  Otherwise, apply \(q\) to the two
admissible presentations
\[
 \prin f=\nabla(\top,f),
 \qquad
 \prin{fg}=\nabla(\prin g,f).
\]
Their inputs have the same image, so
\(q(\prin{fg})=q(\prin f)=\top\), proving (S3).  Therefore \(W\)
defines the singleton-generated topology \(J_W\), and the cokernel
uniqueness argument identifies the original extensional presentation with
\eqref{eq:variation-short-exact}.  Finally, a singleton-generated
topology \(J\) is recovered from
\(W_J=\{f\mid\prin f\in J\}\), while
\eqref{eq:principal-cover-test} recovers \(W\) from \(J_W\).  The two
constructions are inverse.  Since the unary distributivity is strong,
Theorem~\ref{thm:ordinary-strong-presentations} identifies the ordinary
and strong versions.
\end{proof}

\subsection{The form of subobjects}
\label{sec:subobject-form}

Let \(\mathscr E\) be a locally small, well-powered
\emph{geometric category}: a regular category in which all small joins
of subobjects exist and are stable under pullback.  We work in suitable
universes; compare \cite[A1.4]{JohnstoneElephant}.  The clusters over
\(X\) are the subobjects of \(X\), and
the form map sends a monomorphism to its codomain.  Pullback supplies
all finite initial lifts.  If \(a_i:A_i\rightarrowtail X_i\) represents
a subobject and \(f_i:X_i\to X\), take as selected final lift
\begin{equation}\label{eq:subobject-final-lift}
 \nabla_i(A_i,f_i)=\bigvee_i\operatorname{im}(f_i a_i).
\end{equation}
Regular images satisfy Frobenius and are stable under pullback, while
unions of subobjects are stable under pullback.  These presentations
therefore define a strong distributivity, called the
\emph{geometric distributivity} of the subobject form
\(\operatorname{Sub}(\mathscr E)\).

Recall that a \emph{universal productive closure operator} is a family
\[
 c_X:\operatorname{Sub}(X)\longrightarrow\operatorname{Sub}(X)
\]
of extensive, monotone, idempotent, finite-meet-preserving maps which
commute with pullback.  We call such an operator a
\emph{Lawvere--Tierney topology in closure form}.  When \(\mathscr E\)
has a subobject classifier, this is equivalent to the usual
Lawvere--Tierney endomorphism of that classifier; in particular this
recovers the standard notion for every Grothendieck topos
\cite{MacLaneMoerdijk,JohnstoneElephant}.  For closure and density in
the broader categorical setting, see
\cite{DikranjanTholen,DuckertsAntoineGranJanelidze}.
Thus, outside the classifier setting, the established name for the
structure used here is \emph{universal productive closure operator};
the qualification ``in closure form'' records exactly this distinction.

\begin{theorem}\label{thm:subobject-topologies}
For a locally small, well-powered geometric category \(\mathscr E\),
Lawvere--Tierney topologies in closure form are naturally in bijection
with extensional presentations, equivalently strong extensional
presentations, of the geometric distributivity form
\(\operatorname{Sub}(\mathscr E)\).  For the closure operator \(c\), a
representative is
\begin{equation}\label{eq:subobject-short-exact}
 \operatorname{Dense}_c
 \longrightarrow\operatorname{Sub}(\mathscr E)
 \xrightarrow{\,q_c\,}\operatorname{Cl}_c(\mathscr E),
\end{equation}
where
\[
 \operatorname{Dense}_c(X)=\{A\mid c_X(A)=X\},
 \qquad q_c(A)=c_X(A),
\]
and \(\operatorname{Cl}_c(\mathscr E)_X\) is the poset of
\(c\)-closed subobjects.  The quotient identifies two subobjects
exactly when their closures agree.
\end{theorem}

\begin{proof}
Let \(c\) be a universal productive closure operator.  Its fixed
points are closed under finite meets and pullback.  Their structured
final lift is obtained by closing the right-hand side of
\eqref{eq:subobject-final-lift}.  We claim that \(q_c\) preserves these
presentations.  Put
\[
                  U=\bigvee_i\operatorname{im}(f_i a_i).
\]
For each \(i\), let \(C=c_X(U)\).  Since
\(A_i\leqslant f_i^*C\), universality and productivity give
\[
 c_{X_i}(A_i)\leqslant c_{X_i}(f_i^*C)
   =f_i^*c_X(C)=f_i^*C.
\]
By the image--pullback adjunction,
\(\operatorname{im}(f_i c_{X_i}(A_i))\leqslant C\).  Taking the union
and then closing proves
\[
 c_X\!\left(\bigvee_i\operatorname{im}(f_i c_{X_i}(A_i))\right)
 \leqslant c_X(U),
\]
and the reverse inequality follows from extensivity.  This is exactly
preservation of the geometric final lift.  Thus \(q_c\) is a
distributivity morphism, and its kernel is
\(\operatorname{Dense}_c\).

It is also the cokernel of that kernel.  Regard
\(A\rightarrowtail c_X(A)\) as a subobject of \(c_X(A)\); it is dense
there, because pullback stability gives
\[
 c_{c_X(A)}(A)=c_X(A).
\]
If a distributivity morphism \(r\) sends all dense subobjects to top,
apply it to the two unary geometric presentations along the inclusion
\(m:c_X(A)\rightarrowtail X\):
\[
 A=\nabla(A\rightarrowtail c_X(A),m),
 \qquad
 c_X(A)=\nabla(\top_{c_X(A)},m).
\]
Their source clusters have the same image, so
\(r(A)=r(c_X(A))\).  Hence \(r\) factors uniquely through \(q_c\), and
the factor preserves the quotient distributivity.  This proves that
\eqref{eq:subobject-short-exact} is an extensional presentation.

Conversely, let the extensional presentation
\[
 D\longrightarrow\operatorname{Sub}(\mathscr E)
 \xrightarrow{q}Q
\]
be given.  Since the geometric distributivity contains arbitrary
vertical unions, every \(q_X\) preserves all joins and therefore has a
right adjoint \(s_X\).  Surjectivity gives \(q_Xs_X=1\), so
\(c_X=s_Xq_X\) is an extensive, idempotent, productive closure
operator.  It remains to verify pullback stability; this also shows that
the fibrewise right adjoints are indexed.

Let \(f:X\to Y\), let \(A\) be a subobject of \(Y\), and put
\(B=c_X(f^*A)\).  Naturality of \(q\) and the adjunction give
\[
                         f^*c_Y(A)\leqslant c_X(f^*A)=B.
\]
For the reverse inequality, preservation of the unary final lift
\(\exists_f\) gives
\[
 q_Y(\exists_fB)=\exists_fq_X(B)
   =\exists_f f^*q_Y(A)\leqslant q_Y(A).
\]
The adjunction \(q_Y\dashv s_Y\), followed by
\(\exists_f\dashv f^*\), yields
\(B\leqslant f^*c_Y(A)\).  Hence
\(c_X(f^*A)=f^*c_Y(A)\).  Finally,
\[
 D_X=\{A\mid q_X(A)=\top\}
    =\{A\mid c_X(A)=X\}.
\]
We have recovered a universal productive closure operator and hence a
Lawvere--Tierney topology in closure form.  The constructions are
inverse by Theorem~\ref{thm:adjoint-closure} and uniqueness of
cokernels.  Strongness follows from Theorem
\ref{thm:ordinary-strong-presentations}.
\end{proof}

Thus the same construction has two classical realizations.  On the form
of sieves it gives Grothendieck topologies on the base; on a geometric
subobject form it gives Lawvere--Tierney topologies in closure form,
expressed through their dense subobjects.

\subsection{Extensional presentations over the terminal base}

Let \((L,\mathcal P)\) be a distributivity form over \(\mathbf1\), so
that \(L\) is a meet-semilattice with top and \(\mathcal P\) is a
chosen substitution-closed class of exact joins.  A pullback-stable
fibrewise filter is now simply a filter \(D\subseteq L\), where the
improper filter is allowed.  Let \(\theta_D\) be the least
distributivity congruence forcing every member of \(D\) to be top.

\begin{theorem}\label{thm:terminal-filter-topologies}
For a filter \(D\subseteq L\), the following conditions are equivalent:
\begin{enumerate}[label=\textup{(\roman*)}]
\item \(D\) is the kernel term of an extensional presentation of
\((L,\mathcal P)\);
\item \(D\) is the top class of a distributivity congruence on \(L\);
\item the top class of \(\theta_D\) is exactly \(D\).
\end{enumerate}
When they hold, the canonical extensional presentation is
\begin{equation}\label{eq:terminal-filter-sequence}
 D\longrightarrow L\longrightarrow L/\theta_D.
\end{equation}
\end{theorem}

\begin{proof}
This is Lemma~\ref{lem:kernel-criterion} in one fibre.  Explicitly, the
kernel of a morphism is its top class, and equality under that morphism
is a distributivity congruence.  Conversely, the top class of a
distributivity congruence is the kernel of the corresponding quotient.
Minimality of \(\theta_D\) identifies the normal-closure condition with
(iii), and Theorem~\ref{thm:cokernels} supplies
\eqref{eq:terminal-filter-sequence}.
\end{proof}

The theorem gives quick computations in familiar algebraic cases.

\begin{theorem}\label{thm:terminal-computations}
The following are extensional presentations of distributivity forms over
\(\mathbf1\); they are automatically strong.
\begin{enumerate}[label=\textup{(\roman*)}]
\item With only unit joins admissible, every filter \(D\) of a
meet-semilattice gives an extensional presentation; the quotient
relation is
\begin{equation}\label{eq:filter-localization}
 x\equiv_Dy
 \quad\Longleftrightarrow\quad
 x\wedge d=y\wedge d\text{ for some }d\in D.
\end{equation}
\item On a bounded distributive lattice with all finite joins
admissible, every lattice filter gives an extensional presentation, and
\eqref{eq:filter-localization} is the usual bounded-distributive-lattice
quotient forcing \(D\) to top.
\item If \(B\) is a \(\kappa\)-complete Boolean algebra, for
\(\kappa\) infinite and regular, with all joins of cardinality less
than \(\kappa\) admissible, its extensional presentations are precisely
the less-than-\(\kappa\)-complete filters.
\item On the powerset frame \(\mathcal P(X)\), with all unions
admissible, the extensional presentations are precisely the principal
filters.  For
\(A\subseteq X\), the corresponding sequence is
\begin{equation}\label{eq:powerset-sequence}
 \{U\subseteq X\mid A\subseteq U\}
 \longrightarrow\mathcal P(X)
 \xrightarrow{\,U\mapsto U\cap A\,}\mathcal P(A)
 .
\end{equation}
\end{enumerate}
\end{theorem}

\begin{proof}
For the unit distributivity, no additional compatibility is imposed on
the meet congruence \eqref{eq:filter-localization}, whose top class is
exactly \(D\).  This proves (i).  In a distributive lattice, finitely
many witnesses in \(D\) may be replaced by their meet, and finite
distributivity makes \eqref{eq:filter-localization} compatible with
finite joins, proving (ii).

For (iii), complementation identifies filters with ideals.  A quotient
preserves the selected joins exactly when the corresponding ideal is
closed under joins of cardinality less than \(\kappa\), equivalently
when the filter is closed under the corresponding meets.  Regularity of
\(\kappa\) supplies the substitution closure.  Finally, an extensional
presentation of \(\mathcal P(X)\) has a complete filter by (iii) with no cardinal bound,
and every complete filter of a powerset is principal, generated by the
intersection of all its members.  Meeting with that generator is the
cokernel in \eqref{eq:powerset-sequence}.
\end{proof}

For a frame with all joins admissible there is a sharper localic
description.  A meet \(\bigwedge_i d_i\) in a frame is called
\emph{strongly exact} when every frame homomorphism preserves it, and a
filter is strongly exact when it is closed under all strongly exact
meets of its members.  For \(d\in L\), let
\(o_d(x)=d\Rightarrow x\) be the open nucleus, and for a filter \(D\)
put
\begin{equation}\label{eq:fitted-nucleus}
                         j_D=\bigvee_{d\in D}o_d
\end{equation}
in the frame of nuclei.

\begin{theorem}\label{thm:frame-topologies}
For a frame \(L\) with all join presentations admissible, extensional
presentations of \(L\) correspond naturally to strongly exact filters of
\(L\), and
hence to fitted sublocales of the locale represented by \(L\).  The
filter \(D\) gives the representative
\[
 D\longrightarrow L
 \xrightarrow{\,j_D\,}\operatorname{Fix}(j_D)
 .
\]
\end{theorem}

\begin{proof}
For the maximal frame distributivity, distributivity congruences are
exactly frame congruences.  The congruence generated by \(d\equiv1\)
is the open congruence
\[
             \Delta_d=\{(x,y)\mid x\wedge d=y\wedge d\}.
\]
Thus \(\theta_D\) is the join of the open congruences \(\Delta_d\), and
on the localic side it represents the intersection of the corresponding
open sublocales, hence a fitted sublocale.  Moshier, Pultr, and Suarez
prove that fitted sublocales correspond to strongly exact filters
\cite{MoshierPultrSuarez}; see also \cite{JaklSuarez}.  Under this
correspondence, \(D\) is strongly exact exactly when it is the top class
of the least frame congruence forcing its members to top.  Theorem
\ref{thm:terminal-filter-topologies} now gives the result, and
\eqref{eq:fitted-nucleus} is the nucleus representing the same quotient
\cite{EscardoNuclei}.
\end{proof}

The restriction to fitted sublocales is essential.  An arbitrary
nucleus need not be determined by its top class.  For example, on the
three-element frame \(0<a<1\), the identity nucleus and the closed
nucleus \(x\mapsto a\vee x\) both have top class \(\{1\}\).  The
cokernel generated by that class is the identity, not the closed
nucleus.  This is precisely the distinction expressed by exactness for
the distributivity after Theorem~\ref{thm:adjoint-closure}.

The preceding applications are geometric and order-theoretic.  We now
turn to two algebraic forms whose clusters are concrete substructures.
With the minimal distributivity their extensional presentations are
determined by pullback-stable filters of substructures.  These filters
act as neighbourhood filters, while pullback stability expresses
continuity of every homomorphism.

\subsection{The form of subgroups of groups}

Let \(\mathbf{Grp}\) be the category of groups and homomorphisms.  The
\emph{form of subgroups}
\[
                       \Sub\longrightarrow\mathbf{Grp}
\]
has the subgroups \(A\leqslant G\) as its clusters over \(G\), and an
arrow \(A\to B\) above \(f:G\to H\) when \(f(A)\leqslant B\).
Consequently,
\[
 \top_G=G,\qquad A\wedge B=A\cap B,\qquad
 f^*B=f^{-1}(B).
\]
Every structured sink has a final lift,
\[
       \nabla_i(A_i,f_i)=
       \left\langle\bigcup_i f_i(A_i)\right\rangle,
\]
but these final lifts need not distribute over intersections.  We equip
\(\Sub\) with the \emph{minimal distributivity}, in which only unit
presentations are admissible.

A topological group is \emph{non-Archimedean} when it has a
neighbourhood basis at the identity consisting of open subgroups; see
Megrelishvili and Shlossberg
\cite{MegrelishviliShlossberg}.  We use their basis condition but,
unlike the standing convention in that reference, do not include
Hausdorffness in the term.  Call an assignment
\(G\mapsto\tau_G\) of such a topology
\emph{functorial} when every group homomorphism is continuous.  This is
the standard use of functorial topology; for the systematic theory in
the abelian case, including the profinite topology, see Dikranjan and
Giordano Bruno \cite{DikranjanGiordanoBruno}.

\begin{theorem}\label{thm:subgroup-nonarchimedean}
Extensional presentations of the minimally distributive subgroup form
are naturally in bijection with functorial non-Archimedean group
topologies.  They are equivalently its strong extensional
presentations.  More
explicitly, both are determined by assignments
\(G\mapsto\mathcal N(G)\) such that:
\begin{enumerate}[label=\textup{(\roman*)}]
\item \(\mathcal N(G)\) is a filter in the lattice of subgroups of
\(G\), where the improper filter is allowed;
\item for every \(f:G\to H\) and \(U\in\mathcal N(H)\), one has
\begin{equation}\label{eq:subgroup-pullback-filter}
                         f^{-1}(U)\in\mathcal N(G).
\end{equation}
\end{enumerate}
Under the equivalence, \(\mathcal N(G)\) is the filter of open
subgroups.  Moreover, \(G\) is Hausdorff in the sense of Definition
\ref{def:hausdorff-object} if and only if the corresponding
topological group is Hausdorff.
\end{theorem}

\begin{proof}
If an extensional presentation of \(\Sub\) has kernel \(K\), then
\(K_G\) is a subgroup filter and the subform condition gives
\eqref{eq:subgroup-pullback-filter}.  Conversely, let
\(\mathcal N\) have properties (i) and (ii).  In the fibre over \(G\)
define
\begin{equation}\label{eq:subgroup-local-relation}
 A\equiv_{\mathcal N,G}B
 \quad\Longleftrightarrow\quad
 A\cap U=B\cap U\text{ for some }U\in\mathcal N(G).
\end{equation}
Intersections of witnesses prove transitivity, and intersection with a
further subgroup proves compatibility with meets.  Taking inverse
images and using (ii) proves that these congruences are indexed.  Since
the distributivity is minimal, they are automatically distributivity
congruences.

Their top classes are exactly the filters \(\mathcal N(G)\).  Indeed,
\(A\equiv_{\mathcal N,G}G\) precisely when \(A\) contains some member
of \(\mathcal N(G)\).  They are also the least indexed meet
congruences with these top classes: if a meet congruence sends \(U\) to
top and \(A\cap U=B\cap U\), then
\[
 [A]=[A]\wedge[U]=[A\cap U]=[B\cap U]
      =[B]\wedge[U]=[B].
\]
Lemma~\ref{lem:kernel-criterion} therefore makes \(\mathcal N\) the
kernel of an extensional presentation of \(\Sub\), and every
extensional presentation arises in this way.

Given \(\mathcal N\), declare \(O\subseteq G\) open when each
\(x\in O\) has \(xU\subseteq O\) for some
\(U\in\mathcal N(G)\).  The filter axioms give a neighbourhood basis.
Condition (ii), applied to inner automorphisms, gives conjugation
invariance of the filter, and hence the group operations are
continuous.  Its open subgroups are exactly the members of
\(\mathcal N(G)\).  Condition (ii) is precisely continuity of every
homomorphism.  Conversely, the open subgroups of a functorial
non-Archimedean topology satisfy (i) and (ii), so the constructions are
inverse.

Finally,
\[
             \rho_{\mathbb E}(G)
                =\bigcap_{U\in\mathcal N(G)}U
\]
is the closure of the identity.  It is trivial exactly when the group
topology is Hausdorff.
The minimal distributivity is strong, so Theorem
\ref{thm:ordinary-strong-presentations} gives the strong assertion.
\end{proof}

Theorem~\ref{thm:subgroup-nonarchimedean} is the main point of the
example: the established notion of a functorial non-Archimedean
topology is precisely an extensional presentation of the subgroup form.

A large family of classical examples is obtained from finite
quotients.  Let \(\mathcal K\) be a nonempty pseudovariety of finite
groups, that is, a class closed under finite products, subgroups, and
quotients, and put
\[
 \mathcal N_{\mathcal K}(G)=
 \{U\leqslant G\mid
   N\leqslant U\text{ for some }N\mathrel{\trianglelefteq}G
   \text{ with }G/N\in\mathcal K\}.
\]

\begin{theorem}\label{thm:pro-K-topology}
The filter \(\mathcal N_{\mathcal K}\) is the kernel of an extensional
presentation of the subgroup form.  Its associated non-Archimedean group topology is
the pro-\(\mathcal K\) topology, and its Hausdorff objects are exactly
the residually \(\mathcal K\) groups.  In particular, the choices of
all finite groups, all finite \(p\)-groups, all finite solvable groups,
and all finite nilpotent groups give respectively the profinite,
pro-\(p\), prosolvable, and pronilpotent topologies with their familiar
residual classes.
\end{theorem}

\begin{proof}
Finite-intersection closure follows because
\[
 G/(N_1\cap N_2)\longrightarrow G/N_1\times G/N_2
\]
is injective.  If \(f:G\to H\) and \(H/N\in\mathcal K\), then
\(G/f^{-1}(N)\) is isomorphic to a subgroup of \(H/N\); hence
\eqref{eq:subgroup-pullback-filter} holds.  Theorem
\ref{thm:subgroup-nonarchimedean} now gives the extensional
presentation and its associated group topology.  Its
separation radical is the intersection of the kernels of all
homomorphisms from \(G\) to members of \(\mathcal K\), which is trivial
exactly under the usual residual condition.
\end{proof}

The profinite and pro-\(p\) topologies and their completions are
standard; see Ribes and Zalesskii \cite{RibesZalesskii}.  In the
profinite case the open subgroups are precisely the finite-index
subgroups, since every such subgroup contains its finite-index normal
core.  In the pro-\(p\) case the correct functorial basis consists of
subgroups containing a normal subgroup with finite \(p\)-group
quotient; a subgroup of \(p\)-power index need not itself have this
property.  On restricting the base to free groups, both the profinite
and the pro-\(p\) examples are Hausdorff on every object.

There is a complementary class-detecting construction.  We use
\emph{preradical} for an assignment
\(R(G)\leqslant G\) satisfying
\begin{equation}\label{eq:functorial-subgroup}
                         f(R(G))\leqslant R(H)
\end{equation}
for every \(f:G\to H\).  Such subgroups are fully invariant and hence
normal.  The normal-subfunctor formulation and the basic commutator
example already appear in Eilenberg and Mac~Lane
\cite[Section~16]{EilenbergMacLane}.

\begin{theorem}\label{thm:preradical-topology}
Every group preradical \(R\) determines an extensional presentation of
the subgroup form whose kernel is the principal filter
\[
             \mathcal N_R(G)=
             \{U\leqslant G\mid R(G)\leqslant U\}.
\]
The induced group topology is the inverse image of the discrete
topology on \(G/R(G)\), and \(G\) is Hausdorff if and only if
\(R(G)=1\).
\end{theorem}

\begin{proof}
Equation \eqref{eq:functorial-subgroup} implies
\eqref{eq:subgroup-pullback-filter}, so Theorem
\ref{thm:subgroup-nonarchimedean} applies.  The open subsets are exactly
the unions of cosets of \(R(G)\), proving the quotient description.
Lemma~\ref{lem:hausdorff-intrinsic} gives the Hausdorff assertion.
\end{proof}

For \(R(G)=[G,G]\), the Hausdorff objects are the abelian groups.  The
choices \(R(G)=\gamma_{c+1}(G)\) and \(R(G)=G^{(d)}\) detect,
respectively, the groups nilpotent of class at most \(c\) and the groups
of derived length at most \(d\).  Verbal subgroups similarly detect
arbitrary varieties of groups.  A downward directed family of
preradicals \((R_i)\) gives the filter of subgroups containing some
\(R_i(G)\); its Hausdorff objects satisfy
\[
                           \bigcap_iR_i(G)=1.
\]
The lower central and derived series therefore recover the usual
residually nilpotent and residually solvable groups.  See Robinson
\cite{Robinson} for the group-theoretic background.

The minimal distributivity makes the classification in Theorem
\ref{thm:subgroup-nonarchimedean} exhaustive.  One may instead select
vertical exact joins in subgroup lattices; directed unions are basic
examples.  If the common-witness relation
\eqref{eq:subgroup-local-relation} respects those joins, then its
minimal-distributivity extensional presentation remains one for the
richer distributivity.  In particular, all principal preradical
examples do so, because intersection with \(R(G)\) preserves every
exact join by definition of exactness.

\subsection{The form of ideals of commutative rings}

Let \(\CRing\) denote the category of commutative rings with identity,
including the zero ring, and identity-preserving homomorphisms.  Its
\emph{form of ideals}
\[
                         \Idl\longrightarrow\CRing
\]
has the ideals \(I\trianglelefteq R\) as clusters over \(R\), with an
arrow \(I\to J\) above \(f:R\to S\) when \(f(I)\subseteq J\).
Thus
\[
 \top_R=R,\qquad \bot_R=0,\qquad I\wedge J=I\cap J,\qquad
 f^*J=f^{-1}(J).
\]
A structured sink has as its final lift the ideal generated by all the
images:
\[
       \nabla_i(I_i,f_i)=
       \left\langle\bigcup_i f_i(I_i)\right\rangle_R.
\]
As with subgroups, arbitrary such final lifts do not distribute over
intersections.  We first give \(\Idl\) the minimal distributivity.

Recall that a topological ring is \emph{linearly topologized} when zero
has a neighbourhood basis consisting of ideals
\cite[Tag~07E8]{StacksProject}.  A \emph{functorial linear ring
topology} on \(\CRing\) will mean a choice of such a topology on every
ring for which every unital homomorphism is continuous.

\begin{theorem}\label{thm:ideal-linear-topologies}
Extensional presentations of the minimally distributive ideal form are
naturally in bijection with functorial linear ring topologies.  They are
equivalently its strong extensional presentations.  Explicitly, they are
assignments \(R\mapsto\mathcal N(R)\) such that
\(\mathcal N(R)\) is a filter of ideals (the improper filter being
allowed) and
\begin{equation}\label{eq:ideal-pullback-filter}
 J\in\mathcal N(S),\quad f:R\longrightarrow S
 \quad\Longrightarrow\quad f^{-1}(J)\in\mathcal N(R).
\end{equation}
The kernel fibre \(\mathcal N(R)\) is the filter of open ideals.  A
ring \(R\) is Hausdorff in the sense of Definition
\ref{def:hausdorff-object} if and only if its associated linear
topology is separated, or equivalently Hausdorff.
\end{theorem}

\begin{proof}
The kernel fibres of an extensional presentation are filters and their inverse-image
stability is exactly \eqref{eq:ideal-pullback-filter}.  Conversely, for
such a family define
\begin{equation}\label{eq:ideal-local-relation}
 I\equiv_{\mathcal N,R}J
 \quad\Longleftrightarrow\quad
 I\cap U=J\cap U\text{ for some }U\in\mathcal N(R).
\end{equation}
The same meet-congruence argument as in the proof of Theorem
\ref{thm:subgroup-nonarchimedean} shows that these are indexed
distributivity congruences, that their top classes are precisely
\(\mathcal N(R)\), and that they are the least congruences with those
top classes.  Lemma~\ref{lem:kernel-criterion} therefore gives a
unique extensional presentation of \(\Idl\), and every extensional
presentation is obtained in this way.

Starting from \(\mathcal N(R)\), take the cosets of its members as a
neighbourhood basis on the additive group of \(R\).  Because the basis
consists of ideals, addition and multiplication are continuous, so this
is a linear ring topology.  Formula
\eqref{eq:ideal-pullback-filter} says exactly that every ring
homomorphism is continuous.  Conversely, the open ideals of a
functorial linear topology form a pullback-stable filter.  These
constructions are inverse.  Finally,
\[
                 \rho_{\mathbb E}(R)
                    =\bigcap_{U\in\mathcal N(R)}U
\]
is the closure of zero, which vanishes exactly when the topology is
Hausdorff.
The minimal distributivity is strong, so Theorem
\ref{thm:ordinary-strong-presentations} gives the strong assertion.
\end{proof}

For completeness, the canonical quotients in the two algebraic
examples identify substructures which agree inside an open
neighbourhood, as in \eqref{eq:subgroup-local-relation} and
\eqref{eq:ideal-local-relation}.  They serve to realize exactness here;
no standard group- or ring-theoretic construction is being asserted.

The finite-quotient examples have a precise analogue.  Let
\(\mathcal K\) be a nonempty pseudovariety of finite commutative rings,
closed under finite products, unital subrings, and quotients.  With our
convention that the zero ring belongs to \(\CRing\), quotient closure
forces the zero ring to belong to \(\mathcal K\).  Define
\begin{equation}\label{eq:ring-pro-K-filter}
 \mathcal N_{\mathcal K}(R)=
 \{U\trianglelefteq R\mid
   I\subseteq U\text{ for some }I\trianglelefteq R
   \text{ with }R/I\in\mathcal K\}.
\end{equation}

\begin{theorem}\label{thm:ring-pro-K-topology}
The filters \(\mathcal N_{\mathcal K}\) determine an extensional
presentation of the ideal form.  The associated linear topology is the
pro-\(\mathcal K\) topology, and its Hausdorff objects are exactly the
residually \(\mathcal K\) rings.  Taking all finite commutative rings
gives the profinite topology.  Every commutative ring finitely
generated as a unital ring, equivalently every commutative
\(\mathbb Z\)-algebra of finite type, is Hausdorff for this topology.
\end{theorem}

\begin{proof}
If \(R/I_1,R/I_2\in\mathcal K\), then
\[
 R/(I_1\cap I_2)\longrightarrow R/I_1\times R/I_2
\]
is injective, so \eqref{eq:ring-pro-K-filter} is closed under finite
intersections.  If \(f:R\to S\) and \(S/J\in\mathcal K\), then
\(R/f^{-1}(J)\) is isomorphic to a unital subring of \(S/J\).
This proves \eqref{eq:ideal-pullback-filter}, and Theorem
\ref{thm:ideal-linear-topologies} applies.  The separation radical is
the intersection of the kernels of all homomorphisms from \(R\) into
members of \(\mathcal K\), giving the residual characterization.  The
last assertion is the classical residual-finiteness theorem for
commutative rings finitely generated as unital rings
\cite{OrzechRibes}; see also the explicit formulation in
\cite[Lemma~2.1]{BellDanchev}.
\end{proof}

Consequently, the profinite example restricts to a functorial topology
on finitely generated commutative rings for which every object is
Hausdorff, while remaining non-discrete on objects such as
\(\mathbb Z\).  For the broader universal-algebraic setting of
profinite topologies, see Almeida and Costa \cite{AlmeidaCosta}.

Adic topologies arise from functorial powers.  To state the construction
once, let \(U:\mathcal A\to\CRing\) be any functor and suppose that each
object \(A\) is equipped with an ideal
\(\mathfrak a_A\trianglelefteq U(A)\) such that
\[
             U(f)(\mathfrak a_A)\subseteq\mathfrak a_B
\]
for every \(f:A\to B\).  Pull back the ideal form along \(U\).

\begin{theorem}\label{thm:adic-form-topology}
The ideals containing some power \(\mathfrak a_A^n\) form the kernel of
an extensional presentation of the pulled-back ideal form.  The
corresponding classical
topology on \(U(A)\) is the \(\mathfrak a_A\)-adic linear topology, and
\(U(A)\) is Hausdorff exactly when
\begin{equation}\label{eq:adic-separated}
                       \bigcap_{n\geqslant0}\mathfrak a_A^n=0.
\end{equation}
In particular:
\begin{enumerate}[label=\textup{(\roman*)}]
\item for a fixed prime \(p\), on \(\CRing\) the choice
\(\mathfrak a_R=pR\) gives the usual
\(p\)-adic topology;
\item on the category of pairs \((R,I)\), with morphisms
\(f:(R,I)\to(S,J)\) satisfying \(f(I)\subseteq J\), the distinguished
ideal gives the usual \(I\)-adic topology;
\item on the category of Noetherian local rings and local
homomorphisms, the maximal-ideal-adic topology is Hausdorff on every
object.
\end{enumerate}
\end{theorem}

\begin{proof}
The powers are downward directed, and functoriality gives
\[
 U(f)(\mathfrak a_A^n)\subseteq\mathfrak a_B^n.
\]
The resulting filters therefore satisfy
\eqref{eq:ideal-pullback-filter}.  The proof of Theorem
\ref{thm:ideal-linear-topologies} applies verbatim over \(\mathcal A\):
it identifies the induced topology on \(U(A)\) with the ordinary adic
topology, while Definition~\ref{def:hausdorff-object} gives
\eqref{eq:adic-separated}.  The first two cases are immediate.
In the third, local homomorphisms preserve maximal ideals and Krull's
intersection theorem gives
\(\bigcap_n\mathfrak m_R^n=0\) for every Noetherian local ring \(R\);
see \cite[Tag~00IP]{StacksProject}.
\end{proof}

There are also principal examples analogous to group preradicals.  A
\emph{functorial ideal assignment} is a choice
\(\mathfrak r(R)\trianglelefteq R\) such that
\[
                         f(\mathfrak r(R))
                              \subseteq\mathfrak r(S)
\]
for every \(f:R\to S\).

\begin{theorem}\label{thm:functorial-ideal-topology}
Every functorial ideal assignment \(\mathfrak r\) determines an
extensional presentation of the ideal form with kernel
\[
 \mathcal N_{\mathfrak r}(R)
   =\{U\trianglelefteq R\mid\mathfrak r(R)\subseteq U\}.
\]
The associated ring topology is pulled back from the discrete topology
on \(R/\mathfrak r(R)\), and its Hausdorff objects are precisely the
rings for which \(\mathfrak r(R)=0\).
\end{theorem}

\begin{proof}
The principal filters satisfy
\eqref{eq:ideal-pullback-filter}, so Theorem
\ref{thm:ideal-linear-topologies} gives the extensional presentation
and its associated ring topology.  Its open sets
are exactly the unions of cosets of \(\mathfrak r(R)\).  The last
assertion is Lemma~\ref{lem:hausdorff-intrinsic}.
\end{proof}

Several familiar classes are obtained in this way.  The nilradical
detects reduced rings.  The ideal \(nR\), for a fixed positive integer
\(n\), detects rings whose characteristic divides \(n\).  The
additive-torsion ideal of \(R\) detects rings with
torsion-free additive group.  Finally, the ideal
\[
                    (\,x^2-x\mid x\in R\,)
\]
detects Boolean rings.  Each assignment is functorial under all unital
homomorphisms.  By contrast, the Jacobson radical is not functorial for
arbitrary homomorphisms: under the inclusion
\(\mathbb Z_{(p)}\hookrightarrow\mathbb Q\), the element \(p\) belongs
to \(J(\mathbb Z_{(p)})\), whereas \(J(\mathbb Q)=0\).  It therefore
gives an extensional presentation by this construction only after the
base category or
its morphisms are suitably restricted.

Filters of ideals also occur in the theory of Gabriel topologies, but
the two notions should not be conflated.  For a commutative ring, every
upward-closed ideal filter is automatically stable under colons, since
\(I\subseteq(I:r)\).  A Gabriel topology additionally satisfies the
Gabriel transitivity, or locality, axiom and corresponds to a
hereditary torsion theory; see Gabriel \cite{Gabriel} and Stenstr\"om
\cite{Stenstrom}.  Conversely, a Gabriel topology is attached to one
ring and carries no cross-ring inverse-image requirement.  For example,
the principal nilradical filter on
\(k[\varepsilon]/(\varepsilon^2)\) arises from the present functorial
construction but is not a Gabriel topology.

As for subgroups, using the minimal distributivity is what gives the
full classification by linear topologies.  Selecting vertical exact
joins imposes a stronger compatibility.  Whenever
\eqref{eq:ideal-local-relation} respects the selected joins, the same
extensional presentation remains one.  The principal examples in
Theorem~\ref{thm:functorial-ideal-topology} satisfy this condition for
all vertical exact joins, because meeting with \(\mathfrak r(R)\)
preserves every exact join.

\section{Final Remarks}
\label{sec:final-remarks}

\subsection{Support and the universal forms}

Recall the support sieve \(\supp[F]{A}\) from
\eqref{eq:support-sieve-early}.  An admissible final presentation
\(U=\nabla_i(A_i,f_i)\) is \emph{tight} when
\[
           \supp[F]{U}=\bigcup_i f_i\supp[F]{A_i},
\]
and it is \emph{top-generated} when all the \(A_i\) are top clusters.

\begin{theorem}\label{thm:universal-sieve-form}
Let \(F\) be a distributivity form over \(\C\).
\begin{enumerate}[label=\textup{(\roman*)}]
\item The support assignments define a morphism of the underlying
finite-initial-lift forms
\[
                         \sigma_F:F\longrightarrow\Sieve.
\]
They define a distributivity morphism if and only if every admissible
presentation of \(F\) is tight.
\item If every admissible presentation is tight and every cluster has
an admissible top-generated presentation, then \(\sigma_F\) is the
unique distributivity morphism \(F\to\Sieve\).
\item There is at most one distributivity morphism
\(\lambda_F:\Sieve\to F\).  If it exists, it is forced to have the
formula
\begin{equation}\label{eq:lambda-formula}
 \lambda_F(R)=\nabla_{f\in R}(\top_{\operatorname{dom}f},f).
\end{equation}
It exists exactly when the displayed final lifts are admissible, the
assignments \(\lambda_F\) preserve finite initial lifts, and they carry
all presentations of the maximal sieve distributivity to admissible
presentations of \(F\).
\end{enumerate}
Consequently, \(F\) is isomorphic to the maximally distributive form
of sieves when the hypotheses of (ii) and (iii) hold simultaneously.
\end{theorem}

\begin{proof}
Support sends top to the maximal sieve and meets to intersections.  For
\(g:Z\to X\),
\[
 h\in\supp[F]{g^*A}
 \quad\Longleftrightarrow\quad
 (gh)^*A=\top
 \quad\Longleftrightarrow\quad
 h\in g^*\supp[F]{A},
\]
so it also preserves inverse images.  On an admissible presentation,
the final lift of the image sink in \(\Sieve\) is
\(\bigcup_i f_i\supp[F]{A_i}\).  Equality with the support of the apex
is exactly tightness, proving (i).

If \(A=\nabla_i(\top,f_i)\) is tight and admissible, every
distributivity morphism \(u:F\to\Sieve\) must satisfy
\[
 u(A)=\bigcup_i\prin{f_i}=\supp[F]{A}.
\]
This proves uniqueness in (ii).  Conversely, applying a morphism out
of \(\Sieve\) to its canonical final presentation
\[
                 R=\nabla_{f\in R}(\top,f)
\]
forces \eqref{eq:lambda-formula}.  The conditions in (iii) say exactly
that the forced assignment is a distributivity morphism.  When both
directions exist, tight top-generation gives
\(\lambda_F\sigma_F=1_F\), and the canonical final presentation of every
sieve gives \(\sigma_F\lambda_F=1_{\Sieve}\).
\end{proof}

There is an analogous unary statement for the form of variations.  If
\(\C\) has pullbacks and \(F\) has a tight unary distributivity, the
support map factors through \(\Var\) exactly when every support sieve is
principal.  If every cluster is top-generated by one arrow, that factor
is unique.  Conversely, at most one distributivity morphism
\(\Var\to F\) exists, and it is forced by
\[
                         \prin f\longmapsto\nabla(\top,f).
\]
It exists precisely when these unary final lifts define a morphism and
preserve the unary distributivities.  Under both sets of hypotheses the
two maps are inverse.  This is the unary specialization of the proof of
Theorem~\ref{thm:universal-sieve-form}.

\subsection{The category obtained by forgetting distributivity}

Let \(\MFrm(\C)\) denote the category of forms over \(\C\) admitting
all finite initial lifts, with morphisms preserving those lifts.  The
trivial form is a zero object and the zero map is fibrewise constant
top.  Unlike the situation for selected final lifts, no zero completion
is needed.

\begin{theorem}\label{thm:underlying-kernels-cokernels}
The category \(\MFrm(\C)\) has all kernels and cokernels.  For
\(u:F\to G\), the kernel has fibres
\[
                    (\ker u)_X=\{A\mid u_X(A)=\top_X\}.
\]
The cokernel is the fibrewise quotient by
\begin{equation}\label{eq:underlying-cokernel-relation}
 B\equiv_u C
 \quad\Longleftrightarrow\quad
 B\wedge u(A)=C\wedge u(A)
 \text{ for some }A\in F_X.
\end{equation}
These congruences are stable under inverse image.
\end{theorem}

\begin{proof}
The kernel fibres contain top and are closed under meets and inverse
images; the universal property is immediate from their definition.
For the cokernel, reflexivity of \eqref{eq:underlying-cokernel-relation}
is witnessed by top and symmetry is immediate.  Two successive
witnesses may be replaced by their meet, proving transitivity.  Meeting
a witnessing equality with a further cluster proves compatibility with
meets.  If \(B\equiv_uC\) is witnessed by \(A\) and \(f:X\to Y\), then
\[
 f^*B\wedge u(f^*A)
 =f^*(B\wedge u(A))
 =f^*(C\wedge u(A))
 =f^*C\wedge u(f^*A),
\]
so inverse image descends.  The quotient kills \(u\).  If \(r:G\to H\)
kills \(u\), applying \(r\) to a witnessing equality gives
\(r(B)=r(C)\), and therefore the required unique factorization.
\end{proof}

The quotient has a useful fraction interpretation.  Put
\[
 N_u(X)=\{U\in G_X\mid u(A)\leqslant U
                    \text{ for some }A\in F_X\}.
\]
For \(U\in N_u(X)\), the quotient inverts every vertical arrow
\[
                         B\wedge U\longrightarrow B.
\]
An arrow from \([B]\) to \([C]\) above \(f:X\to Y\) is represented by
a roof
\[
 B\ \longleftarrow\ B\wedge U\ \longrightarrow\ C,
 \qquad U\in N_u(X),
\]
and exists exactly when \(B\wedge U\leqslant f^*C\).  Denominators
have common refinements by meets, and composition refines \(U,V\) to
\(U\wedge f^*V\).  The distributivity cokernel of
Theorem~\ref{thm:cokernels} is obtained by further closing this
fibrewise fraction relation under all admissible final presentations.

A full subform \(K\hookrightarrow F\) is normal in \(\MFrm(\C)\) when
its fibres are pullback-stable filters.  The normal monomorphisms and
normal epimorphisms are respectively the kernels and cokernels just
described.

\begin{theorem}\label{thm:underlying-homological}
Normal monomorphisms and normal epimorphisms in \(\MFrm(\C)\) are
closed under composition, and \(\MFrm(\C)\) satisfies the subquotient
axiom.  It is therefore a pointed homological category in the sense of
Grandis, and in particular a Grandis ex2-category
\cite{GrandisBook,JanelidzeWeighillII}.
\end{theorem}

\begin{proof}
A filter in a filter is a filter in the ambient meet-semilattice, so
normal monomorphisms compose.  Given successive quotient maps
\[
 F\xrightarrow{q_K}F/K\xrightarrow{q_L}(F/K)/L,
\]
the inverse image \(M=q_K^{-1}(L)\) is normal.  Lifting a witness in
\(L\) along the surjection \(q_K\), and then meeting it with a witness
for equality in the first quotient, gives
\((F/K)/L\cong F/M\).  Thus normal epimorphisms compose.

For the subquotient axiom, let \(B\hookrightarrow A\) be normal and let
\(q_K:A\to A/K\) be normal, with \(K\subseteq B\).  The image
\(q_K(B)\) is closed under top, meets, and inverse images.  It is upward
closed: if \(q_K(b)\leqslant q_K(a)\), then for some \(k\in K\) one has
\(b\wedge k\leqslant a\).  Since \(b,k\in B\), upward closure gives
\(a\in B\).  The restriction \(B\to q_K(B)\) is the quotient of
\(B\) by \(K\), while \(q_K(B)\hookrightarrow A/K\) is normal.  This
is the required normal-epi--normal-mono factorization.
\end{proof}

No analogous homological claim holds for
\((\DFrm(\C),\Null)\), or for its strong counterpart.  They are
semiexact by Theorem~\ref{thm:cokernels}, and their cokernels form the
left class of the orthogonal factorization system in Theorem
\ref{thm:cokernel-factorization-system}.  Nevertheless its right class
is strictly larger than the monomorphisms, cokernels are not stable
under pullback by Theorem~\ref{thm:cokernels-not-pullback-stable}, and
Theorem~\ref{thm:composition-normal-maps} shows that kernels need not
compose.  The failure on the monomorphic side is precisely where the
chosen distributivity differs from the underlying meet-form calculus.

The difference between the underlying and distributive cokernels is
already visible on a three-object category.  Let
\[
                     0\xrightarrow{e}1\xrightarrow{f}2
\]
be a chain, and let \(J\) be the singleton-generated topology in which
\(e\) covers but \(f\) does not.  The underlying cokernel in
\(\MFrm(\C)\) identifies \(\prin e\) with \(\top_1\), but leaves
\(\prin{fe}\) distinct from \(\prin f\) over \(2\).  It cannot preserve
the two admissible presentations
\[
 \prin f=\nabla(\top_1,f),
 \qquad
 \prin{fe}=\nabla(\prin e,f).
\]
The distributivity cokernel propagates the first identification through
this final lift and therefore identifies the two outputs.  Equivalently,
\(f^*\prin{fe}=\prin e\) covers, so the two principal sieves have the
same local saturation.  This small example isolates why the cokernel in
Theorem~\ref{thm:main} must remember admissible presentations.

\subsection{Further extensional presentations}

Several standard Grothendieck topologies give especially transparent
quotients.

\begin{example}[Open covers]
Let \(X\) be a topological space and take
\(\C=\mathcal O(X)\), ordered by inclusion.  A sieve \(S\) on an open
set \(U\) covers when the union of its members is \(U\), and
\[
                  j(S)=\prin{\bigcup S\hookrightarrow U}.
\]
Thus \(\Sat_J\) is isomorphic, as an underlying form, to the form of
variations, but it carries the full distributivity of structured unions
of opens rather than only the unary one.  A representative is
\[
 J_{\mathrm{open}}\longrightarrow\Sieve
 \xrightarrow{\,S\mapsto\prin{\bigcup S\hookrightarrow U}\,}
 \Var_{\mathrm{open}}.
\]
The quotient forgets the particular sieve and remembers its union.
\end{example}

\begin{example}[Nested topologies]
If \(J\subseteq K\) are Grothendieck topologies on \(\C\), restriction
of \(J\)-saturation gives
\[
 J\longrightarrow K
 \longrightarrow\{S\in K\mid j_J(S)=S\}.
\]
The right-hand form carries the quotient distributivity.  The local
presentations used in the proof of Theorem~\ref{thm:main} remain inside
\(K\), so the same cokernel argument applies.  On a suitably small
category of schemes, inclusions from the Zariski topology to the
\'{e}tale topology and then to the fppf topology give familiar
instances.
\end{example}

\begin{example}[The atomic topology]
Suppose that \(\C\) satisfies the right Ore condition.  Its nonempty
sieves form the atomic topology.  The saturation of the empty sieve is
empty, whereas every nonempty sieve saturates to the maximal sieve.
Hence the associated quotient has two clusters in each fibre and the
extensional presentation is
\[
 J_{\mathrm{at}}\longrightarrow\Sieve
 \longrightarrow\mathbf2,
\]
where \(\mathbf2\) carries the quotient distributivity.
\end{example}

\begin{example}[Stable factorization systems]
Suppose that \(\C\) has pullbacks and a proper factorization system
\((\mathcal E,\mathcal M)\), with \(\mathcal E\) stable under
pullback.  The \(\mathcal E\)-arrows form a saturated singleton
coverage.  Indeed, identities, pullback stability, and composition are
standard.  If \(fg\in\mathcal E\) and \(f=me\) is its factorization,
orthogonality makes \(m\) a split epimorphism; properness also makes it
a monomorphism, hence an isomorphism, so \(f\in\mathcal E\).  Theorem
\ref{thm:singleton-topologies} gives
\[
 \Var_{\mathcal E}\longrightarrow\Var
 \longrightarrow\operatorname{Sub}_{\mathcal M}(\C)
 .
\]
If \(f=me\), its quotient class is represented by \(\prin m\).  For a
regular category this is the sequence from variations generated by
regular epimorphisms to ordinary subobjects, and the quotient records
the image of an arrow.
\end{example}

\subsection{Semilattice sites, partial frames, and points}

The terminal-base examples of Sections~\ref{sec:distributivity-forms}
and \ref{sec:applications} have a direct site-theoretic interpretation.
Regard a meet-semilattice \(L\) as a thin category, with an arrow
\(a\to u\) when \(a\leqslant u\), and declare
\((a_i\to u)_i\) to cover when \(u=\bigvee_i a_i\) is an admissible
presentation.

\begin{lemma}\label{lem:semilattice-sites}
Conditions (D1), (D3), and (D2) are respectively the identity,
pullback, and transitivity axioms for a family-based Grothendieck
coverage on the thin category \(L\).  The resulting coverage is
subcanonical.  Conversely, every subcanonical family-based coverage on
\(L\) determines a distributivity by taking its covering families as
exact join presentations.
\end{lemma}

\begin{proof}
Identity covers are (D1).  Pullback in a meet-semilattice is meet, so
stability of a cover \(u=\bigvee_i a_i\) along \(c\leqslant u\) is
precisely
\[
                         c=\bigvee_i(c\wedge a_i),
\]
which is (D3).  Substitution of covering families is exactly (D2).
A representable presheaf on a poset satisfies descent for
\((a_i\to u)_i\) exactly when \(u\) is the join of the \(a_i\), so the
coverage is subcanonical.  The converse reads these implications in
reverse; pullback stability makes the resulting joins exact.
\end{proof}

Stubbe proves that all distributive joins give the canonical, or
finest subcanonical, Grothendieck topology on a meet-semilattice
\cite{Stubbe}.  Lemma~\ref{lem:maximal-exact-distributivity} is the
family-based version of that result.  Ball and Pultr's exact join sites
\cite{BallPultr} are especially close to our terminal-base objects.
Zenk's partial frames \cite{Zenk}, and the work on their congruences in
\cite{FrithSchauerteMadden,SchauerteFrith}, use a selector fixed across
a category of semilattices; here the distributivity is selected
object-by-object and is required from the outset to be closed under
substitution and meet restriction.  Related selected-join approaches
include \(Z\)-continuous posets \cite{BandeltErne}, \(\kappa\)-frames
\cite{Madden}, and distributive envelopes
\cite{GehrkeVanGool}.

A proper filter \(D\subseteq L\) is \emph{\(\mathcal P\)-prime} when
\[
 u=\bigvee_i a_i\text{ admissible and }u\in D
 \quad\Longrightarrow\quad
 a_i\in D\text{ for some }i.
\]

\begin{lemma}\label{lem:prime-points}
The characteristic map
\[
 \chi_D:L\longrightarrow\mathbf2,
 \qquad
 \chi_D(x)=1\Longleftrightarrow x\in D,
\]
to the maximally distributive two-element frame is a distributivity
morphism if and only if \(D\) is \(\mathcal P\)-prime.  Every
\(\mathcal P\)-prime filter is therefore the kernel of a morphism and
hence the kernel term of an extensional presentation of \(L\).
\end{lemma}

\begin{proof}
The characteristic map preserves top and binary meets exactly when
\(D\) is a filter.  Since every input of a join lies below its apex,
preservation of an admissible join is equivalent to the displayed
primeness condition.  Its kernel is \(D\), which is therefore normal
by Lemma~\ref{lem:kernel-criterion}.
\end{proof}

For the unit distributivity these are all proper filters; for finite
joins they are the ordinary prime filters; for countable joins they are
countably completely prime filters; and for a frame they are completely
prime filters, hence locale points.  The two-element map need not be the
cokernel of its kernel.  On the three-element chain, for instance,
\(\{1\}\) is prime but its canonical cokernel is the identity of the
chain rather than the characteristic map.

The literal form of sieves over \(\mathbf1\) is only \(\mathbf2\).
Accordingly, Theorem~\ref{thm:main} recovers the two Grothendieck
topologies on the terminal category, while the many filters above give
extensional presentations of other middle distributivity forms.  This
distinction is one of the principal benefits of the relative
definition.

\subsection{General one-object bases}

A one-object category is the delooping \(\mathbf BM\) of a monoid
\(M\).  A form over \(\mathbf BM\) admitting finite initial lifts is a
meet-semilattice \(L\) with top and a contravariant \(M\)-action by
finite-meet homomorphisms.  An admissible presentation has the form
\[
                         u=\nabla_i(a_i,m_i),
 \qquad a_i\leqslant m_i^*u,
\]
with \(u\) least subject to these inequalities.  An extensional
presentation of such a form has an \(M\)-stable filter as its kernel,
with the additional
normality imposed by the labelled presentations.  For the unit
distributivity every \(M\)-stable filter is normal, since
\eqref{eq:filter-localization} is stable under the action: a witness
\(d\) is carried to the witness \(m^*d\).

The sieve form over \(\mathbf BM\) is the lattice of right ideals of
\(M\), using the convention that multiplication records composition,
and
\[
                         m^*R=\{n\in M\mid mn\in R\}.
\]
Theorem~\ref{thm:main} identifies its extensional presentations with the
ordinary Grothendieck topologies on \(\mathbf BM\).  Its variations are the
principal right ideals.  They form a distributivity form with the unary
calculus only when the pullback hypothesis needed in
Theorem~\ref{thm:singleton-topologies} is available; the sieve theorem
requires no such hypothesis.  Thus nontrivial one-object bases give an
equivariant refinement of the semilattice theory rather than merely
another copy of the terminal-base case.

\subsection{Towards relative sheaf theory}

An ordinary Grothendieck topology is useful not only because it singles
out covering sieves, but because those sieves determine a category of
sheaves.  The preceding results therefore suggest a further test for the
notion of an extensional presentation: it should support a corresponding
relative sheaf theory.  We give here a preliminary construction.  Its
decisive feature is that it recovers both ordinary presheaves and ordinary
sheaves when the middle term is the form of sieves.

Let \(F\) be a distributivity form with total category
\(p:\mathcal E\to\C\), and write
\(y:\mathcal E\to[\mathcal E^{\mathrm{op}},\mathbf{Set}]\) for the
Yoneda embedding.  An admissible presentation
\[
                         \pi:\quad
                         U=\nabla_i(A_i,f_i)
\]
determines total arrows \(u_i:A_i\to U\).  Let
\(R_\pi\hookrightarrow yU\) be the sieve generated by this family.

\begin{definition}\label{def:relative-presheaf}
An \emph{\(F\)-presheaf} is a presheaf
\(P:\mathcal E^{\mathrm{op}}\to\mathbf{Set}\) for which the canonical
map
\[
 P(U)\cong\operatorname{Nat}(yU,P)
       \longrightarrow\operatorname{Nat}(R_\pi,P)
\]
is a bijection for every admissible presentation \(\pi\).  We denote the
full category of such presheaves by \(\operatorname{PSh}(F)\).
\end{definition}

Thus an \(F\)-presheaf has unique gluing for the presentations built into
the distributivity form.  Using the generated sieve, rather than a
pairwise equalizer formula, makes the definition meaningful without
assuming that the relevant pullbacks exist.  The word ``presheaf'' records
that this descent belongs to the ambient form; an extensional presentation
will impose the additional sheaf condition.

Now let
\[
 \mathbb X:\qquad K\longrightarrow F\xlongrightarrow{q}Q
\]
be an extensional presentation.  Let \(W_{\mathbb X}\) consist of the
vertical arrows
\[
                  w_{A,B}:A\longrightarrow B,
 \qquad A\leqslant B,\qquad q_X(A)=q_X(B).
\]
These are precisely the vertical inclusions made locally invisible by the
quotient.  In particular, since \(K=\ker q\), a cluster \(A\in F_X\)
belongs to \(K_X\) exactly when
\(A\to\top_X\) belongs to \(W_{\mathbb X}\).

\begin{definition}\label{def:relative-sheaf}
A \emph{sheaf relative to \(\mathbb X\)}, or an
\emph{\(\mathbb X\)-sheaf}, is an \(F\)-presheaf \(P\) such that
\[
                         P(B)\longrightarrow P(A)
\]
is a bijection for every \(w_{A,B}\in W_{\mathbb X}\).  The resulting
full category is denoted \(\operatorname{Sh}_{\mathbb X}(F)\).
\end{definition}

This definition depends only on the isomorphism class of the extensional
presentation fixing \(F\).  It could equivalently be expressed by saying
that \(P\) is orthogonal, in the presheaf category, to the inclusions
\(R_\pi\hookrightarrow yU\) and to the maps
\(yA\to yB\) induced by the arrows in \(W_{\mathbb X}\).

\begin{lemma}\label{lem:relative-sheaf-categories}
Suppose that the total category \(\mathcal E\) is small.  Then the
categories \(\operatorname{PSh}(F)\) and
\(\operatorname{Sh}_{\mathbb X}(F)\) are reflective locally presentable
categories.  Suppose in addition that \(\C\) has pullbacks and \(F\) is
strong.  Then \(\operatorname{PSh}(F)\) is the Grothendieck topos of
sheaves for the topology on \(\mathcal E\) generated by the admissible
presentations.  Moreover, \(\operatorname{Sh}_{\mathbb X}(F)\) is the
subtopos obtained by adjoining the principal sieves generated by
\(W_{\mathbb X}\) as covering sieves.
\end{lemma}

\begin{proof}
Because \(\mathcal E\) is small, each of the two definitions is an
orthogonality class determined by a set of maps in the locally presentable
category \([\mathcal E^{\mathrm{op}},\mathbf{Set}]\).  Small
orthogonality classes are reflective and locally presentable
\cite{AdamekRosicky}.

If \(F\) is strong, Theorem~\ref{thm:strong-total-site} says that its
admissible families form a Grothendieck pretopology on \(\mathcal E\).
The first orthogonality condition is therefore exactly the usual sheaf
condition for the topology it generates.  The class
\(W_{\mathbb X}\) contains identities and is closed under composition.
It is also stable under pullback in \(\mathcal E\).  Indeed, the pullback
of \(A\to B\) along an arrow \(C\to B\) above \(f:Y\to X\) is
\[
                       C\wedge f^*A\longrightarrow C.
\]
Since \(C\leqslant f^*B\) and \(q\) preserves inverse images and finite
meets,
\[
 q(C\wedge f^*A)
   =q(C)\wedge f^*q(A)
   =q(C)\wedge f^*q(B)
   =q(C).
\]
Thus the pulled-back arrow again lies in \(W_{\mathbb X}\).  The
admissible families together with these singleton families consequently
generate a Grothendieck topology, and the second orthogonality condition
is its sheaf condition.  The inclusion of the resulting sheaf category
into \(\operatorname{PSh}(F)\) is therefore a subtopos inclusion.
\end{proof}

For a distributivity which is not strong, the first conclusion of the
lemma is the appropriate general statement: the relative presheaf
category need not be a topos, since the admissible presentations need not
be stable under cartesian base change.  One may always pass instead to the
least Grothendieck topology containing the sieves \(R_\pi\).  This gives a
topos, but it is a completion of the original distributivity and may
impose descent for presentations not selected in \(F\).

The principal example shows that the terminology has the intended
meaning.

\begin{theorem}[Sieve comparison]\label{thm:relative-sieve-sheaves}
For the maximally distributive form of sieves on \(\C\), there is a
natural equivalence
\[
             \operatorname{PSh}(\Sieve)\simeq
             [\C^{\mathrm{op}},\mathbf{Set}].
\]
If \(J\) is a Grothendieck topology and
\[
             \mathbb X_J:\qquad
             J\longrightarrow\Sieve\longrightarrow\Sat_J
\]
is its extensional presentation from Theorem~\ref{thm:main}, this
equivalence restricts to
\[
             \operatorname{Sh}_{\mathbb X_J}(\Sieve)\simeq
             \operatorname{Sh}(\C,J).
\]
\end{theorem}

\begin{proof}
Let \(\mathcal E_{\Sieve}\) be the total category of \(\Sieve\), and
consider the fully faithful top-cluster section
\[
 t:\C\longrightarrow\mathcal E_{\Sieve},
 \qquad X\longmapsto\top_X.
\]
Every sieve \(S\) on \(X\) has its canonical admissible presentation
\[
             S=\nabla_{f\in S}
               (\top_{\operatorname{dom}f},f).
\]
Consequently every object of \(\mathcal E_{\Sieve}\) is covered by
objects in the image of \(t\).  The topology induced on \(\C\) is the
trivial topology: a family of top clusters presents \(\top_X\) exactly
when its arrows generate the maximal sieve.  The Comparison Lemma
\cite[Chapter~III, Section~4]{MacLaneMoerdijk} now gives the first
equivalence.  Explicitly, a presheaf \(H\) on \(\C\) corresponds to
\[
                       \widetilde H(S)
                          =\operatorname{Nat}(S,H),
\]
where \(S\) is regarded as a subpresheaf of \(yX\); in particular,
\(\widetilde H(\top_X)=H(X)\).

Write \(j_J\) for \(J\)-saturation.  A vertical inclusion
\(R\to S\) belongs to \(W_{\mathbb X_J}\) precisely when
\(j_J(R)=j_J(S)\), that is, precisely when \(R\) is \(J\)-dense in
\(S\).  In particular, for \(S\in J(X)\), the condition associated
with \(S\to\top_X\) is the bijectivity of
\[
                 H(X)\longrightarrow\operatorname{Nat}(S,H),
\]
which is the ordinary \(J\)-sheaf condition.  Conversely, a
\(J\)-sheaf is orthogonal to every \(J\)-dense inclusion of sieves.
The second equivalence follows.
\end{proof}

\begin{example}[Locales and discrete spaces]\label{ex:relative-locale-sheaves}
Let \(L\) be a frame, regarded as a distributivity form over
\(\mathbf1\) with every join presentation admissible.  Its total site
is the usual site of the locale represented by \(L\); hence
\(\operatorname{PSh}(L)\) is the ordinary category of sheaves on that
locale.  For the extensional presentation
\[
       D\longrightarrow L
        \xlongrightarrow{j_D}\operatorname{Fix}(j_D)
\]
of Theorem~\ref{thm:frame-topologies}, the arrows inverted by the
quotient generate the topology of the fitted sublocale represented by
\(j_D\).  Therefore
\[
       \operatorname{Sh}_{\mathbb X_D}(L)\simeq
       \operatorname{Sh}(\operatorname{Fix}(j_D)).
\]

The powerset instance is especially concrete.  For \(A\subseteq X\),
the extensional presentation \eqref{eq:powerset-sequence} gives
\[
       \operatorname{PSh}(\mathcal P(X))\simeq\mathbf{Set}^{X},
       \qquad
       \operatorname{Sh}_{\mathbb X_A}(\mathcal P(X))
          \simeq\mathbf{Set}^{A}.
\]
Thus the extensional presentation literally restricts sheaves on the
discrete space \(X\) to sheaves on the selected subspace \(A\).
\end{example}

Theorem~\ref{thm:relative-sieve-sheaves} suggests two further problems.
First, one should characterize those distributivity forms for which the
sheafified top-cluster section recovers the original form as a pullback
of the subobject form.  This requires both separation of clusters and a
comprehension condition asserting that every relevant subobject is
represented by a cluster.  Second, it would be useful to identify
relative sheaf categories arising from the algebraic examples of
Section~\ref{sec:applications}.  The additive-subgroup form, for
example, naturally expresses descent along additive generation and
should not be expected by itself to recover schemes.  A connection with
scheme theory would instead require a suitably oriented ideal or
radical-ideal form, Zariski localization covers, and an additional local
representability condition.

\end{document}